\documentclass[a4paper,11pt]{article}%
\usepackage{amsfonts}
\usepackage{amsmath}
\usepackage{geometry}
\usepackage{hyperref}
\usepackage{amssymb}
\usepackage{graphicx}
\usepackage{caption}%
\providecommand{\U}[1]{\protect\rule{.1in}{.1in}}
\newtheorem{theorem}{Theorem}[section]

\newtheorem{claim}[theorem]{Assumption}

\newtheorem{definition}{Definition}[section]

\newtheorem{proposition}{Proposition}[section]

\newenvironment{proof}[1][Proof]{\noindent\textbf{#1.} }{\ \rule{0.5em}{0.5em}}
\begin{document}

\title{Nonlinear Luenberger-like observers for some anthracnose models}
\author{David Jaures FOTSA MBOGNE$^{a}$\thanks{Corresponding author Address: email:
mjdavidfotsa@gmail.com, P.O. Box 455, ENSAI, The University of Ngaoundere} ,
Duplex Elvis HOUPA DANGA$^{b}$, David BEKOLLE$^{b}$\thanks{Co-authors emails:
e\_houpa@yahoo.com, bekolle@yahoo.fr}\\$^{a}$Department of Mathematics and Computer Science, ENSAI, The University of Ngaoundere\\$^{b}$Department of Mathematics and Computer Science, Fac.Sci., The University
of Ngaoundere }
\date{}
\maketitle

\begin{abstract}
In this paper, we propose observers for the dynamics of anthracnose disease
described in \cite{fotsa}. Spatial and non spatial versions of the observers
are given following an approach similar to \cite{fotsa}. The models given in
\cite{fotsa} are improved in order to be more realistic. The conditions on
parameters are more general. There are also changes in the equations modelling
the dynamics of the berry volume $\left(  v\right)  $ and the rot volume
$\left(  v_{r}\right)  $. We study theoretically the proposed observers in
terms of well-posedness and convergence. We make several simulations to assess
the effectiveness of observers which definitely display fairly good behaviour.



\end{abstract}

\section{Introduction}

\qquad Anthracnose is a phytopathology which occurs on several commercial
tropical crops. The coffee is concerned by that disease due to the
\textit{Colletotrichum} \textit{kahawae} which is an ascomycete fungus
\cite{bieysse,boisson,chen,jeffries,mouen09,muller70,wharton}. Several models
of the athracnose dynamics have been proposed in the literature
\cite{danneberger,dodd,duthie,mouen07,mouen09,mouen072,mouen03,mouen08,wharton}
in order to better understand and control it. Recently, in \cite{fotsa,fotsa2}
the authors studied the optimal control of a diffusion model of anthracnose
with continuous and impulsive strategies. Optimal controls were computed with
respect to given cost functionals. The general model proposed in \cite{fotsa}
was given by the following equations%

\begin{equation}
\partial_{t}\theta=\alpha\left(  t,x\right)  \left(  1-w\left(  t,x\right)
\theta\right)  +\operatorname{div}\left(  A\left(  t,x\right)  \nabla
\theta\right)  ,\text{ on }%
%TCIMACRO{\U{211d} }%
%BeginExpansion
\mathbb{R}
%EndExpansion
_{+}^{\ast}\times U \label{ModelSpace1}%
\end{equation}%
\begin{equation}
\left\langle A\left(  t,x\right)  \nabla\theta\left(  t,x\right)  ,n\left(
x\right)  \right\rangle =0,\text{ on }%
%TCIMACRO{\U{211d} }%
%BeginExpansion
\mathbb{R}
%EndExpansion
_{+}^{\ast}\times\partial U \label{ModelSpace2}%
\end{equation}
where $n\left(  x\right)  $ denotes the normal vector on the boundary at $x$
and
\begin{equation}
\theta\left(  0,x\right)  =\rho\left(  x\right)  \in\left[  0,1\right[
,\text{ }x\in\overline{U}\subseteq%
%TCIMACRO{\U{211d} }%
%BeginExpansion
\mathbb{R}
%EndExpansion
^{3} \label{ModelSpace3}%
\end{equation}%
\begin{equation}
\partial_{t}v=\beta\left(  t,x,\theta\right)  \left(  1-v/\left(  \left(
1-\theta\right)  \eta\left(  t,x\right)  v_{\max}\right)  \right)
\label{ModelSpace4}%
\end{equation}%
\begin{equation}
\partial_{t}v_{r}=\gamma\left(  t,x,\theta\right)  \left(  1-v_{r}/v\right)
\label{ModelSpace5}%
\end{equation}%
\begin{equation}
\left(  v\left(  0,x\right)  ,v_{r}\left(  0,x\right)  \right)  \in\left]
0,v_{\max}\right]  \times\left[  0,v_{\max}\right]  ,\text{ }x\in\overline
{U}\subseteq%
%TCIMACRO{\U{211d} }%
%BeginExpansion
\mathbb{R}
%EndExpansion
^{3} \label{ModelSpace6}%
\end{equation}
where $w\left(  t,x\right)  =1/\left(  1-\sigma u\left(  t,x\right)  \right)
$ and $v\left(  0\right)  \geq v_{r}\left(  0\right)  $.

In the model above, $\theta$ is the inhibition rate, $v$ and $v_{r}$ are
respectively the fruit volume density and the rot volume density. The function
$v$ is upper bounded by a value $v_{\max}$ which models the natural fact that
the fruit growth is limited. Nonnegative functions $\alpha,\beta,\gamma:%
%TCIMACRO{\U{211d} }%
%BeginExpansion
\mathbb{R}
%EndExpansion
_{+}\times U\times%
%TCIMACRO{\U{211d} }%
%BeginExpansion
\mathbb{R}
%EndExpansion
\rightarrow%
%TCIMACRO{\U{211d} }%
%BeginExpansion
\mathbb{R}
%EndExpansion
_{+}$ characterize the effects of environmental and climatic conditions on the
rate of change of inhibition rate, fruit volume, and infected fruit volume
respectively \cite{danneberger,dodd,duthie}. There is a control $u$
representing the chemical strategy consisting on the effects after application
of fungicides. The positive term $1-\sigma\in\left]  0,1\right[  $ models the
inhibition rate corresponding to epidermis penetration by hyphae. Once the
epidermis has been penetrated, the inhibition rate cannot fall below
$1-\sigma$, even under the maximum control effort $(u=1)$. In the absence of
control effort $(u=0)$, the inhibition rate should increase towards $1$. The
effects of environmental and climatic conditions on the maximum fruit volume
are represented by the parameter $\eta$ which is a $\left]  0,1\right]
-$valued function. There is a spatial spread of the disease in the open domain
$U\subset\mathbb{R}^{3}$ of class $C^{1}$ through the diffusion term
$\operatorname{div}\left(  A\nabla\theta\right)  $. The boundary condition
$\left\langle A\nabla\theta,n\right\rangle =0$ where $A$ is a $3\times
3$-matrix $\left(  a_{ij}\right)  $ could be viewed as the law steering
migration of the disease between $U$ and its exterior. For instance, if $A$ is
reduced to the identity matrix $I$ then $\left\langle \nabla\theta
,n\right\rangle =0$ means that the domain $U$ does not have any exchange with
its exterior.

Also in \cite{fotsa} a within-host model has been studied and given by the
following equations :%

\begin{equation}
d_{t}\theta=\alpha\left(  t\right)  \left(  1-w\left(  t\right)
\theta\right)  \label{ModelIntraHote1}%
\end{equation}%
\begin{equation}
d_{t}v=\beta\left(  t,\theta\right)  \left(  1-v/\eta\left(  t\right)
v_{\max}\left(  1-\theta\right)  \right)  \label{ModelIntraHote2}%
\end{equation}%
\begin{equation}
d_{t}v_{r}=\gamma\left(  t,\theta\right)  \left(  1-v_{r}/v\right)
\label{ModelIntraHote3}%
\end{equation}%
\begin{equation}
\left(  \theta\left(  0\right)  ,v\left(  0\right)  ,v_{r}\left(  0\right)
\right)  \in\left[  0,1\right[  \times\left]  0,v_{\max}\right]  \times\left[
0,v_{\max}\right]  \label{ModelIntraHote4}%
\end{equation}
where $w\left(  t\right)  =1/\left(  1-\sigma u\left(  t\right)  \right)  $
and $v\left(  0\right)  \geq v_{r}\left(  0\right)  $. The parameters in
$\left(  \ref{ModelIntraHote1}\right)  -\left(  \ref{ModelIntraHote4}\right)
$ still have the same interpretations given in the general model $\left(
\ref{ModelSpace1}\right)  -\left(  \ref{ModelSpace6}\right)  $.

\qquad In several cases, especially for the results in \cite{fotsa,fotsa2} on
anthracnose disease, an optimal control strategy is given as a feedback and
needs to know the state of the system and parameters values. It could be
difficult to know exactly the dynamics of the inhibition rate while it is
easier to observe volumes $v$ and $v_{r}$. The aim of this paper is to provide
some observers for the state $\theta$ in the models $\left(  \ref{ModelSpace1}%
\right)  -\left(  \ref{ModelSpace6}\right)  $ and $\left(
\ref{ModelIntraHote1}\right)  -\left(  \ref{ModelIntraHote4}\right)  $
assuming that volumes $v$ and $v_{r}$ are known. An observer\footnote{See
\cite{trinh12} for more details on the topic.} is a function $\widehat{\theta
}$ having its own dynamics depending on the available observations ($v$ and
$v_{r}$) and converging towards a neighborhood of $\theta$. Let us recall the
following useful adapted definitions about the stability of an observer. The
authors refers to \cite{coron,luenberger,trinh12,tuscnak} and references
therein for more general and classical definitions.

\begin{definition}
Assume that $\theta$ is $S$-valued. An observer $\widehat{\theta}$ is said to be

\begin{enumerate}
\item[$\left(  i\right)  $] (Locally) stable if there are two neighborhoods
$U,V\subseteq S$ of $0$ and a time $t_{0}>0$ such that if $\theta\left(
0\right)  -\widehat{\theta}\left(  0\right)  \in U$ then $\forall t>t_{0}$,
$\theta\left(  t\right)  -\widehat{\theta}\left(  t\right)  \in V$;

\item[$\left(  ii\right)  $] (Locally) asymptotically stable if there is a
neighborhood $U\subseteq S$ of $0$ such that if $\theta\left(  0\right)
-\widehat{\theta}\left(  0\right)  \in U$ then $\underset{t\rightarrow\infty
}{\lim}\theta\left(  t\right)  -\widehat{\theta}\left(  t\right)  =0$;

\item[$\left(  iii\right)  $] (Locally) exponentially asymptotically stable if
there are a neighborhood $U\subseteq S$ of $0$, a positive real constant $k$
and a time $t_{0}>0$ such $\forall t>t_{0}$, $\left\Vert \theta\left(
t\right)  -\widehat{\theta}\left(  t\right)  \right\Vert <\exp\left(
-kt\right)  $.
\end{enumerate}

Those local stabilitiy properties are said to be global if one can extend the
set $U$ to the whole set $S$.
\end{definition}

\qquad The remainder of the work is organized as it follows. In section
\ref{SectionNewModelling} we present some new realistic considerations made
for the models. We also show that the new model is well-posed. We design
observers for the non spatial model in section \ref{SectionObserverODE}. A
theoretical study is done in subsection \ref{SubsectionDesignTheoric1} and we
make numerical evaluations in subsection \ref{ObserversEvalutionNum1}. We also
design observers for the spatial model in section \ref{SectionObserverPDE} and
we study them theoretically in subsection \ref{SubsectionDesignTheoric2}. A
numerical evaluation is carried out in subsection \ref{ObserversEvalutionNum2}%
. Finally, there is a conclusion in section \ref{Conclusion}.

\section{New modelling of anthracnose dynamics\label{SectionNewModelling}}

\qquad In this section we introduce some changes in the models given in
\cite{fotsa} based on some remarks in order to become more realistic. Our
modifications are especially made on the description of the dynamics of $v$
and $v_{r}$.

Looking at equation $\left(  \ref{ModelSpace4}\right)  $ we can see that when
$\theta$ tends to the value $1$, $v$ tends with an infinite speed to zero.
However, we think that either there is a disease or not the berries will reach
and remain greater than a minimal volume. Indeed, the disease can not be
triggered when the berry is less than a minimal size as for instance the bud
size. Even if the fruit totally rots it keeps a minimal size. We then suggest
to add a very small nonegative term $\varepsilon\ll1$ in such way that
equation $\left(  \ref{ModelSpace4}\right)  $ becomes
\begin{equation}
\partial_{t}v=\beta\left(  t,x,\theta\right)  \left(  1-v/\left(  \left(
1+\varepsilon-\theta\right)  \eta\left(  t,x\right)  v_{\max}\right)  \right)
\label{ModelSpace42}%
\end{equation}
With the same arguments equation $\left(  \ref{ModelIntraHote2}\right)  $
becomes%
\begin{equation}
d_{t}v=\beta\left(  t,\theta\right)  \left(  1-v/\eta\left(  t\right)
v_{\max}\left(  1+\varepsilon-\theta\right)  \right)  \label{ModelIntraHote22}%
\end{equation}
The use of the term $\varepsilon$ generalizes the model in such way that we
can recover the original model given in \cite{fotsa} by setting $\varepsilon
=0$. With a positive $\varepsilon$ even if $\theta=1,$ $v$ might be also
positive. Indeed, the berries reach a minimal volume in the neighborhood of
$\varepsilon v_{\max}\underset{t}{\min}\left\{  \eta\left(  t\right)
\right\}  $.

We now focus on the equations $\left(  \ref{ModelSpace5}\right)  $ and
$\left(  \ref{ModelIntraHote3}\right)  $. In \cite{fotsa} it was established
that $v_{r}$ is $\left[  0,v_{\max}\right]  $ valued. However, that behaviour
is not sufficient. Indeed, $v_{r}$ should remain less than $v$. On the other
hand $v_{r}$ should remain equal to zero while $v$ is less than the minimal
value corresponding to the threshold volume where the disease can be
triggered. We then suggest that $v_{r}$ remains equal to zero when $v$ is not
greater than $0$. Instead of studiying directly $v_{r}$, we introduce a
proportion function $\rho$ depending on time and space variables, and given
such as
\begin{equation}
v_{r}=\rho v
\end{equation}

We propose the following dynamics for the rot proportion $\rho$: $\forall
x\in\overline{U}$, $\rho\left(  0,x\right)  \in\left[  0,1\right]  $ and%

\begin{equation}
\partial_{t}\rho\left(  t,x\right)  =\overline{\gamma}\left(  t,x,\theta
\left(  t,x\right)  ,v\left(  t,x\right)  ,\rho\left(  t,x\right)  \right)
\left(  1-\rho\left(  t,x\right)  \right)  ,\text{ }\forall\left(  t,x\right)
\in%
%TCIMACRO{\U{211d} }%
%BeginExpansion
\mathbb{R}
%EndExpansion
_{+}^{\ast}\times\overline{U}\text{,} \label{ModelSpace52}%
\end{equation}
where $\overline{\gamma}:%
%TCIMACRO{\U{211d} }%
%BeginExpansion
\mathbb{R}
%EndExpansion
_{+}\times U\times%
%TCIMACRO{\U{211d} }%
%BeginExpansion
\mathbb{R}
%EndExpansion
^{3}\rightarrow%
%TCIMACRO{\U{211d} }%
%BeginExpansion
\mathbb{R}
%EndExpansion
$ is a function satisfiying $\forall\left(  t,x\right)  \in%
%TCIMACRO{\U{211d} }%
%BeginExpansion
\mathbb{R}
%EndExpansion
_{+}\times U$,

\begin{claim}
\label{Hypothesis_1}$\overline{\gamma}$ is a measurable with respect to the
two first parameters and locally Lipschitz continuous with respect to the
third parameter.
\end{claim}

\begin{claim}
\label{Hypothesis_2}$\overline{\gamma}\left(  t,x,.,.,0\right)  $ is
nonnegative and such that
\begin{equation}
\overline{\gamma}\left(  t,x,0,.,0\right)  =0=\overline{\gamma}\left(
t,x,.,0,.\right)  \text{.}%
\end{equation}

\end{claim}

\begin{claim}
\label{Hypothesis_3}$\overline{\gamma}\left(  t,x,.,.,.\right)  $ is
increasing with respect to the first parameter and $\overline{\gamma}\left(
t,x,0,.,.\right)  $ is decreasing with respect to the last parameter.
\end{claim}

For instance $\overline{\gamma}$ could be chosen with the simple but general
form $\overline{\gamma}\left(  t,x,y_{1},y_{2},y_{3}\right)  =\left(
\gamma_{1}\left(  t,x\right)  y_{1}-\gamma_{2}\left(  t,x\right)
y_{3}\right)  y_{2}$, with $\gamma_{1},\gamma_{2}:%
%TCIMACRO{\U{211d} }%
%BeginExpansion
\mathbb{R}
%EndExpansion
_{+}\times U\times%
%TCIMACRO{\U{211d} }%
%BeginExpansion
\mathbb{R}
%EndExpansion
^{3}\rightarrow%
%TCIMACRO{\U{211d} }%
%BeginExpansion
\mathbb{R}
%EndExpansion
_{+}$. If we omit the dependence with the space variable we get the
corresponding equation for the non spatial model: $\rho\left(  0\right)
\in\left[  0,1\right]  $ and
\begin{equation}
d_{t}\rho\left(  t\right)  =\overline{\gamma}\left(  t,\theta\left(  t\right)
,v\left(  t\right)  ,\rho\left(  t\right)  \right)  \left(  1-\rho\left(
t\right)  \right)  ,\text{ }\forall t\in%
%TCIMACRO{\U{211d} }%
%BeginExpansion
\mathbb{R}
%EndExpansion
_{+}^{\ast}\text{.} \label{ModelIntraHote32}%
\end{equation}

As we will check in the sequel, the function $\rho$ is $\left[  0,1\right]
$-valued. The assumption $\left(  \ref{Hypothesis_2}\right)  $ guarantees that
while the berry has a null volume (without berry) or the disease has not
started the rot volume remains null. The assumption $\left(
\ref{Hypothesis_3}\right)  $ means that the rot volume increases with the
inhibition rate. When there is a not inhibition the volume of rot does not
increase while the fruit grows better and therefore the proportion $\rho$ decreases.

\subsection{Well-posedness of the problem for the non spatial model}

\qquad This subsection is devoted to the proof of the well-posedness of the
non spatial model. We make the following useful assumptions.

\begin{claim}
\label{Hypothesis_4}$\alpha\in L^{\infty}\left(
%TCIMACRO{\U{211d} }%
%BeginExpansion
\mathbb{R}
%EndExpansion
_{+};\left[  0,1\right[  \right)  $.
\end{claim}

\begin{claim}
\label{Hypothesis_5}$u,\eta\in L^{\infty}\left(
%TCIMACRO{\U{211d} }%
%BeginExpansion
\mathbb{R}
%EndExpansion
_{+};\left[  0,1/\left(  1+\varepsilon\right)  \right]  \right)  $ and
$\forall t\geq0,$ $\eta\left(  t\right)  \geq\eta^{\ast}\in\left]  0,1\right[
$.
\end{claim}

\begin{claim}
\label{Hypothesis_6}$\beta:%
%TCIMACRO{\U{211d} }%
%BeginExpansion
\mathbb{R}
%EndExpansion
_{+}\times U\times%
%TCIMACRO{\U{211d} }%
%BeginExpansion
\mathbb{R}
%EndExpansion
\rightarrow%
%TCIMACRO{\U{211d} }%
%BeginExpansion
\mathbb{R}
%EndExpansion
_{+}$ is nonincreasing function with respect to the third parameter. $\forall
z\in%
%TCIMACRO{\U{211d} }%
%BeginExpansion
\mathbb{R}
%EndExpansion
$, $\beta\left(  .,.,z\right)  $ is a measurable function and belongs to the
space $L_{loc}^{\infty}\left(
%TCIMACRO{\U{211d} }%
%BeginExpansion
\mathbb{R}
%EndExpansion
_{+}\times%
%TCIMACRO{\U{211d} }%
%BeginExpansion
\mathbb{R}
%EndExpansion
;%
%TCIMACRO{\U{211d} }%
%BeginExpansion
\mathbb{R}
%EndExpansion
\right)  $.
\end{claim}

\begin{claim}
\label{Hypothesis_7}$\forall t\geq0$, the function $\beta\left(  t,.\right)  $
is differentiable and $\partial_{x}\beta\left(  t,.\right)  \in L_{loc}%
^{\infty}\left(
%TCIMACRO{\U{211d} }%
%BeginExpansion
\mathbb{R}
%EndExpansion
;%
%TCIMACRO{\U{211d} }%
%BeginExpansion
\mathbb{R}
%EndExpansion
\right)  $.
\end{claim}

The assumptions $\left(  \ref{Hypothesis_4}\right)  $-$\left(
\ref{Hypothesis_6}\right)  $ are more general than those given in \cite{fotsa}
in the sense that all the parameters of the model were assumed continuous with
respect to the state variable while here they are just measurable and
essentially bounded. The following existence and uniqueness propositiont holds:

\begin{proposition}
\label{ExistenceOdeModel}If $\left(  \theta\left(  0\right)  ,v\left(
0\right)  ,\rho\left(  0\right)  \right)  \in$ $\left[  0,1\right]
\times\left[  0,v_{\max}\right]  \times\left[  0,1\right]  $ then the problem
$\left(  \ref{ModelIntraHote1}\right)  $, $\left(  \ref{ModelIntraHote22}%
\right)  $ and $\left(  \ref{ModelIntraHote32}\right)  $ has a unique absolute
continuous solution $\left(  \theta,v,\rho\right)  $ valued in the set
$\left[  0,1\right]  \times\left[  0,v_{\max}\right]  \times\left[
0,1\right]  $.
\end{proposition}

\begin{proof}
Using the Carath\'{e}odory theorem (see \cite{coddington}) we get easily the
existence of an absolutely continuous solution. The logistic form of the
equations $\left(  \ref{ModelIntraHote1}\right)  $ and $\left(
\ref{ModelIntraHote22}\right)  $ shows clearly that for a maximal solution
$\left(  \theta,v,v_{r}\right)  $ the couple $\left(  \theta,v\right)  $ is
necessarily valued in the set $\left[  0,1\right]  \times\left]  0,v_{\max
}\right]  $. On the other hand from the equation $\left(
\ref{ModelIntraHote32}\right)  $, $\rho$ is valued in $\left[  0,1\right]  $.
To show it let set $\rho_{1}=\max\left\{  0,-\rho\right\}  $ and $\rho
_{2}=\max\left\{  0,\rho-1\right\}  $. The functions $\rho_{1}$ and $\rho_{2}$
are continuous as $\rho$ is continuous. Without loss of generality we will
assume that $\rho_{1}$ and $\rho_{2}$ are positive on an open interval of time
$\left]  0,T\right[  $. Then $\forall t\in\left]  0,T\right[  $,
\begin{align}
\rho_{1}\left(  t\right)   &  =-\int\nolimits_{0}^{t}\overline{\gamma}\left(
s,\theta\left(  s\right)  ,v\left(  s\right)  ,-\rho_{1}\left(  s\right)
\right)  \left(  1+\rho_{1}\left(  s\right)  \right)  ds\\
&  \leq0\nonumber
\end{align}%
\begin{align}
\rho_{2}\left(  t\right)   &  =-\int\nolimits_{0}^{t}\overline{\gamma}\left(
s,\theta\left(  s\right)  ,v\left(  s\right)  ,\rho_{2}\left(  s\right)
+1\right)  \rho_{2}\left(  s\right)  ds\\
&  \leq0\nonumber
\end{align}
There is a contradiction meaning that $\rho_{1}$ and $\rho_{2}$ are
identically null. Therefore, $\rho$ is $\left[  0,1\right]  $-valued.
\end{proof}

\subsection{Well-posedness of the problem for the spatial model}

\qquad In this subsection the well-posedness of the spatial model is proved.
To that aim, we make some assumptions.

\begin{claim}
\label{Hypothesis_8}$\alpha\in L_{loc}^{\infty}\left(
%TCIMACRO{\U{211d} }%
%BeginExpansion
\mathbb{R}
%EndExpansion
_{+};L^{\infty}\left(  U;%
%TCIMACRO{\U{211d} }%
%BeginExpansion
\mathbb{R}
%EndExpansion
_{+}\right)  \right)  $.
\end{claim}

\begin{claim}
\label{Hypothesis_9}$u,\eta\in L^{\infty}\left(
%TCIMACRO{\U{211d} }%
%BeginExpansion
\mathbb{R}
%EndExpansion
_{+};L^{\infty}\left(  U;\left[  0,1\right]  \right)  \right)  $ and $\forall
t\geq0,$ $1/\left(  1+\varepsilon\right)  \geq\eta\left(  t,.\right)  \geq
\eta^{\ast}\in L^{\infty}\left(  U;\left[  0,1\right]  \right)  $ with
$0<\underset{U}{\inf}\left(  \eta^{\ast}\right)  \overset{def}{=}\eta
_{m}^{\ast}$.
\end{claim}

\begin{claim}
\label{Hypothesis_10}$\beta:%
%TCIMACRO{\U{211d} }%
%BeginExpansion
\mathbb{R}
%EndExpansion
_{+}\times U\times%
%TCIMACRO{\U{211d} }%
%BeginExpansion
\mathbb{R}
%EndExpansion
\rightarrow%
%TCIMACRO{\U{211d} }%
%BeginExpansion
\mathbb{R}
%EndExpansion
_{+}$ is nonincreasing function with respect to the third parameter. $\forall
z\in%
%TCIMACRO{\U{211d} }%
%BeginExpansion
\mathbb{R}
%EndExpansion
$, $\beta\left(  .,.,z\right)  $ is a measurable function and belongs to the
space $L_{loc}^{\infty}\left(
%TCIMACRO{\U{211d} }%
%BeginExpansion
\mathbb{R}
%EndExpansion
_{+}\times%
%TCIMACRO{\U{211d} }%
%BeginExpansion
\mathbb{R}
%EndExpansion
;%
%TCIMACRO{\U{211d} }%
%BeginExpansion
\mathbb{R}
%EndExpansion
_{+}\right)  $.
\end{claim}

\begin{claim}
\label{Hypothesis_11}$\forall i,j\in\left\{  1,2,3\right\}  ,$ $a_{ij}\in
L_{loc}^{\infty}\left(
%TCIMACRO{\U{211d} }%
%BeginExpansion
\mathbb{R}
%EndExpansion
_{+};W^{1,\infty}\left(  U;%
%TCIMACRO{\U{211d} }%
%BeginExpansion
\mathbb{R}
%EndExpansion
\right)  \right)  $.
\end{claim}

\begin{claim}
\label{Hypothesis_12}$\exists\zeta\in%
%TCIMACRO{\U{211d} }%
%BeginExpansion
\mathbb{R}
%EndExpansion
_{+}^{\ast}$ such that $\forall t\in%
%TCIMACRO{\U{211d} }%
%BeginExpansion
\mathbb{R}
%EndExpansion
_{+},\forall w\in H^{1}\left(  U;%
%TCIMACRO{\U{211d} }%
%BeginExpansion
\mathbb{R}
%EndExpansion
\right)  ,$
\[%
%TCIMACRO{\dint \nolimits_{U}}%
%BeginExpansion
{\displaystyle\int\nolimits_{U}}
%EndExpansion
\left\langle A\left(  t,x\right)  \nabla w\left(  x\right)  ,\nabla w\left(
x\right)  \right\rangle dx\geq\zeta%
%TCIMACRO{\dint \nolimits_{U}}%
%BeginExpansion
{\displaystyle\int\nolimits_{U}}
%EndExpansion
\left\langle \nabla w\left(  x\right)  ,\nabla w\left(  x\right)
\right\rangle dx.
\]

\end{claim}

\begin{claim}
\label{Hypothesis_13}$\forall t\geq0,\forall x\in U,$ the function
$\beta\left(  t,x,.\right)  $ is differentiable and $\partial_{y}\beta\left(
t,x,.\right)  \in L_{loc}^{\infty}\left(
%TCIMACRO{\U{211d} }%
%BeginExpansion
\mathbb{R}
%EndExpansion
;%
%TCIMACRO{\U{211d} }%
%BeginExpansion
\mathbb{R}
%EndExpansion
\right)  $.
\end{claim}

Assumptions $\left(  \ref{Hypothesis_8}\right)  $-$\left(  \ref{Hypothesis_10}%
\right)  $\textbf{ }are also more general than those given in \cite{fotsa} in
the sense that all the parameters of the model were assumed continuous with
respect to the state variable while here they are just measurable and
essentially bounded. From assumptions $\left(  \ref{Hypothesis_11}\right)
$-$\left(  \ref{Hypothesis_12}\right)  $ the following problem has a unique
solution in $H^{1}\left(  U\right)  $ for an arbitrary but fixed time $t>0$.%
\[
\left\{
\begin{array}
[c]{l}%
\operatorname{div}\left(  A\left(  t,x\right)  \nabla w\left(  x\right)
\right)  =f\left(  x\right)  ,\text{ }\forall x\in U\\
\left\langle A\left(  t,x\right)  \nabla w\left(  x\right)  ,n\left(
x\right)  \right\rangle =0,\text{ }\forall x\in\partial U
\end{array}
\right.
\]
where $f\in L^{2}\left(  U\right)  $. Following \cite{barbu} in Theorems 3.6.1
and 3.6.2, there is an orthonormal complete system $\left\{  \varphi
_{n}\left(  t,.\right)  \right\}  _{n\in%
%TCIMACRO{\U{2115} }%
%BeginExpansion
\mathbb{N}
%EndExpansion
}\subset L^{2}\left(  U\right)  $ of eingenfunctions and eingenvalues
$\left\{  \lambda_{n}\left(  t\right)  \right\}  $ such that $\forall n\in%
%TCIMACRO{\U{2115} }%
%BeginExpansion
\mathbb{N}
%EndExpansion
,$
\[
\left\{
\begin{array}
[c]{l}%
\operatorname{div}\left(  A\left(  t,x\right)  \nabla\varphi_{n}\left(
t,x\right)  \right)  =\lambda_{n}\left(  t\right)  \varphi_{n}\left(
t,x\right)  ,\text{ }\forall x\in U\\
\left\langle A\left(  t,x\right)  \nabla\varphi_{n}\left(  t,x\right)
,n\left(  x\right)  \right\rangle =0,\text{ }\forall x\in\partial U
\end{array}
\right.  .
\]
Moreover, $\left\{  \varphi_{n}\left(  t,.\right)  \right\}  _{n\in%
%TCIMACRO{\U{2115} }%
%BeginExpansion
\mathbb{N}
%EndExpansion
}\subset H^{1}\left(  U\right)  $ and if $\partial U$ is of class $C^{2}$ then
$\left(  \varphi_{n}\left(  t,.\right)  \right)  _{n\in%
%TCIMACRO{\U{2115} }%
%BeginExpansion
\mathbb{N}
%EndExpansion
}$ is $H^{2}\left(  U\right)  -$valued.

Now we make the following additional assumption.

\begin{claim}
\label{Hypothesis_14}The sequence $\left(  \varphi_{n}\right)  $ does not
depends on the time (ie $\varphi_{n}\left(  t,.\right)  =\varphi_{n}\left(
.\right)  ,$ $\forall t>0$).
\end{claim}

Assumption $\left(  \ref{Hypothesis_14}\right)  $ could happen if $A\left(
t,.\right)  $ has the form $\mu\left(  t\right)  B\left(  .\right)  $ with
$\mu\left(  t\right)  \in%
%TCIMACRO{\U{211d} }%
%BeginExpansion
\mathbb{R}
%EndExpansion
$, $\forall t\geq0$. It will be the case in particular if $A\left(
t,.\right)  =\mu\left(  t\right)  I$ and therefore $\operatorname{div}\left(
A\left(  t,.\right)  \nabla w\right)  =\mu\left(  t\right)  \Delta w$. Here
$I$ denotes the identity matrix of $%
%TCIMACRO{\U{211d} }%
%BeginExpansion
\mathbb{R}
%EndExpansion
^{3}$. Whether $\left(  \ref{Hypothesis_14}\right)  $ is satisfied a weak
solution $\theta$ of $\left(  \ref{ModelSpace1}\right)  -\left(
\ref{ModelSpace3}\right)  $ can be written as the sum $%
%TCIMACRO{\dsum \nolimits_{n=0}^{\infty}}%
%BeginExpansion
{\displaystyle\sum\nolimits_{n=0}^{\infty}}
%EndExpansion
\theta_{n}\varphi_{n}$ where each $\theta_{n}$ is an absolutely continuous
function of the time and satisfies $\forall t\geq0,$
\begin{align}
&  \theta_{n}\left(  t\right)  =%
%TCIMACRO{\dsum \nolimits_{m=0}^{\infty}}%
%BeginExpansion
{\displaystyle\sum\nolimits_{m=0}^{\infty}}
%EndExpansion
\theta_{m}%
%TCIMACRO{\dint \nolimits_{U}}%
%BeginExpansion
{\displaystyle\int\nolimits_{U}}
%EndExpansion
\varphi_{m}\left(  x\right)  \varphi_{n}\left(  x\right)
dx\label{DecomposedSolution}\\
&  =\theta_{n}\left(  0\right)  +%
%TCIMACRO{\dint \nolimits_{0}^{t}}%
%BeginExpansion
{\displaystyle\int\nolimits_{0}^{t}}
%EndExpansion
\lambda_{n}\left(  s\right)  \theta_{n}\left(  s\right)  ds+%
%TCIMACRO{\dint \nolimits_{0}^{t}}%
%BeginExpansion
{\displaystyle\int\nolimits_{0}^{t}}
%EndExpansion%
%TCIMACRO{\dint \nolimits_{U}}%
%BeginExpansion
{\displaystyle\int\nolimits_{U}}
%EndExpansion
\alpha\left(  s,x\right)  \varphi_{n}\left(  x\right)  dsdx\nonumber\\
&  -%
%TCIMACRO{\dint \nolimits_{0}^{t}}%
%BeginExpansion
{\displaystyle\int\nolimits_{0}^{t}}
%EndExpansion%
%TCIMACRO{\dint \nolimits_{U}}%
%BeginExpansion
{\displaystyle\int\nolimits_{U}}
%EndExpansion
\alpha\left(  s,x\right)  w\left(  s,x\right)  \theta\left(  s,x\right)
\varphi_{n}\left(  x\right)  dxds\text{.}\nonumber
\end{align}

\begin{proposition}
The problem $\left(  \ref{ModelSpace1}\right)  -\left(  \ref{ModelSpace3}%
\right)  $ has a unique solution $\theta\in C\left(
%TCIMACRO{\U{211d} }%
%BeginExpansion
\mathbb{R}
%EndExpansion
_{+};H^{1}\left(  U;\left[  0,1\right]  \right)  \right)  $ which is
absolutely continuous with respect to the time.
\end{proposition}

\begin{proof}
It is sufficient here to etablish that if there is a maximal solution then
that solution is valued in the set $\left[  0,1\right]  $. Indeed, conditions
$\left(  \ref{Hypothesis_8}\right)  $-$\left(  \ref{Hypothesis_12}\right)  $
are sufficient to use the Carath\'{e}odory theorem (see in \cite{anita} the
methods used in the proof of the Lemma A 2.7, p187-191). Let $\forall t\geq0,$
$\theta_{1}=\max\left\{  0,-\theta\right\}  $ and $\theta_{2}=\max\left\{
0,\theta-1\right\}  $. We have $\theta_{1}\left(  0,.\right)  =\theta
_{2}\left(  0,.\right)  =0$. Let $S_{i}$ be a convex open subset of $%
%TCIMACRO{\U{211d} }%
%BeginExpansion
\mathbb{R}
%EndExpansion
_{+}$ where functions $\theta_{i}\left(  t,.\right)  $ ($i\in\left\{
1,2\right\}  $) are positive on subsets of $U$ with a positive measure. Then
$\forall t\in S_{i}$ we have%
\begin{align*}
\partial_{t}\left\Vert f_{1}\left(  t,.\right)  \right\Vert _{L^{2}\left(  U;%
%TCIMACRO{\U{211d} }%
%BeginExpansion
\mathbb{R}
%EndExpansion
\right)  }^{2}/2  &  =\int\nolimits_{U}f_{1}\left(  t,x\right)  \partial
_{t}f_{1}\left(  t,x\right)  dx\\
&  =-\int\nolimits_{U}\alpha\left(  t,x\right)  \left(  1+w\left(  t,x\right)
f_{1}\left(  t,x\right)  \right)  f_{1}\left(  t,x\right)  dx\\
&  \leq0
\end{align*}%
\begin{align*}
\partial_{t}\left\Vert f_{2}\left(  t,.\right)  \right\Vert _{L^{2}\left(  U;%
%TCIMACRO{\U{211d} }%
%BeginExpansion
\mathbb{R}
%EndExpansion
\right)  }^{2}/2  &  =\int\nolimits_{U}f_{2}\left(  t,x\right)  \partial
_{t}f_{2}\left(  t,x\right)  dx\\
&  =\int\nolimits_{U}\alpha\left(  t,x\right)  \left(  1-w\left(  t,x\right)
\right)  f_{2}\left(  t,x\right)  dx-\int\nolimits_{U}\alpha\left(
t,x\right)  w\left(  t,x\right)  f_{2}^{2}\left(  t,x\right)  dx\\
&  \leq0
\end{align*}
Using the above inequalities we get that functions $\theta_{i}\left(
t,.\right)  $ are nonpositive on the sets $S_{i}$ which are necessary empty.
We conclude that $\forall t>0,$ $\theta\left(  t,.\right)  $ is valued in
$\left[  0,1\right]  $.
\end{proof}

\begin{proposition}
The problem $\left(  \ref{ModelSpace42}\right)  ,\left(  \ref{ModelSpace52}%
\right)  $ has a unique solution $\left(  v,\rho\right)  \in C\left(
%TCIMACRO{\U{211d} }%
%BeginExpansion
\mathbb{R}
%EndExpansion
_{+};L^{\infty}\left(  U;\left[  0,v_{\max}\right]  \times\left[  0,1\right]
\right)  \right)  $ which is absolutely continuous with respect to the time.
\end{proposition}

\begin{proof}
Since $\theta$ is valued in the set $\left[  0,1\right]  $ and there is not
diffusion in the equations $\left(  \ref{ModelSpace42}\right)  $ and $\left(
\ref{ModelSpace52}\right)  $, we can fix the space variable and use a proof
similar to the one given for proposition \ref{ExistenceOdeModel}.
\end{proof}

\section{Observation for the within-host model\label{SectionObserverODE}}

\subsection{Theoretical design of observers\label{SubsectionDesignTheoric1}}

\qquad The main objective of this subsection is to design observers for the
problem $\left(  \ref{ModelIntraHote1}\right)  $, $\left(
\ref{ModelIntraHote22}\right)  $ and $\left(  \ref{ModelIntraHote32}\right)  $
in order to estimate the inhibition rate $\theta$. Let us consider the
following system : $\forall t\geq0$,
\begin{equation}
d_{t}\widehat{\theta}\left(  t\right)  =\alpha\left(  t\right)  \left(
1-w\left(  t\right)  \widehat{\theta}\left(  t\right)  \right)  +k_{1}\left(
t\right)  \phi_{1}\left(  t,\widehat{\theta}\left(  t\right)  ,\widehat
{v}\left(  t\right)  \right)  +k_{2}\left(  t\right)  \phi_{2}\left(
t,\widehat{\theta}\left(  t\right)  \right)  \label{ObserverODE11}%
\end{equation}
and%
\begin{equation}
d_{t}\widehat{v}\left(  t\right)  =\beta\left(  t,\widehat{\theta}\left(
t\right)  \right)  \phi_{3}\left(  t,\widehat{\theta}\left(  t\right)
,\widehat{v}\left(  t\right)  \right)  \label{ObserverODE12}%
\end{equation}
where
\[
\phi_{1}\left(  t,x,y\right)  =\left\{
\begin{array}
[c]{l}%
\left(  1-v\left(  t\right)  /y\right)  \left(  1+\varepsilon-x\right)
\text{, if }v\left(  t\right)  \leq y\text{ and }\left(  x,y\right)
\in\left]  0,1\right[  \times%
%TCIMACRO{\U{211d} }%
%BeginExpansion
\mathbb{R}
%EndExpansion
_{+}^{\ast}\\
0\text{, otherwise}%
\end{array}
\right.  ,
\]%
\[
\phi_{2}\left(  t,x\right)  =\left\{
\begin{array}
[c]{l}%
d_{t}\rho\left(  t\right)  -\overline{\gamma}\left(  t,x,v\left(  t\right)
,\rho\left(  t\right)  \right)  \left(  1-\rho\left(  t\right)  \right)
\text{, if }x\in\left]  0,1\right[ \\
0\text{, otherwise}%
\end{array}
\right.
\]
and
\[
\phi_{3}\left(  t,x,y\right)  =1-\frac{y}{\left(  1+\varepsilon-x\right)
\eta\left(  t\right)  v_{\max}}\text{.}%
\]

We make the following additional hypothesis :

\begin{claim}
\label{Hypothesis_15}$k_{1},k_{2}\in L_{loc}^{\infty}\left(
%TCIMACRO{\U{211d} }%
%BeginExpansion
\mathbb{R}
%EndExpansion
_{+};%
%TCIMACRO{\U{211d} }%
%BeginExpansion
\mathbb{R}
%EndExpansion
_{+}\right)  $.
\end{claim}

By a solution of the system $\left(  \ref{ObserverODE11}\right)  -\left(
\ref{ObserverODE12}\right)  $ we mean an absolutely continuous function of the
time $\left(  \widehat{\theta},\widehat{v}\right)  $ which satisfies $\forall
t>0,$%
\begin{align*}
\widehat{\theta}\left(  t\right)   &  =\widehat{\theta}\left(  0\right)
+\int\nolimits_{0}^{t}\alpha\left(  s\right)  \left(  1-w\left(  s\right)
\widehat{\theta}\left(  s\right)  \right)  +\int\nolimits_{0}^{t}k_{1}\left(
s\right)  \phi_{1}\left(  s,\widehat{\theta}\left(  s\right)  ,\widehat
{v}\left(  s\right)  \right)  ds\\
&  +\int\nolimits_{0}^{t}k_{2}\left(  s\right)  \phi_{2}\left(  s,\widehat
{\theta}\left(  s\right)  \right)  ds
\end{align*}
and%
\[
\widehat{v}\left(  t\right)  =\widehat{v}\left(  0\right)  +\int
\nolimits_{0}^{t}\beta\left(  s,\widehat{\theta}\left(  s\right)  \right)
\phi_{3}\left(  s,\widehat{\theta}\left(  s\right)  ,\widehat{v}\left(
s\right)  \right)  ds\text{.}%
\]

We have the existence and uniqueness

\begin{proposition}
\label{ExistenceUniciteOb1}If $\left(  \widehat{\theta}\left(  0\right)
,\widehat{v}\left(  0\right)  \right)  \in$ $\left[  0,1\right]  \times\left[
0,v_{\max}\right]  $ then the problem $\left(  \ref{ObserverODE11}\right)
-\left(  \ref{ObserverODE12}\right)  $ has a unique solution valued in
$\left[  0,1\right]  \times\left[  0,v_{\max}\right]  $.
\end{proposition}

\begin{proof}
We first show that if $\left(  \widehat{\theta},\widehat{v}\right)  $ is a
local solution of the system $\left(  \ref{ObserverODE11}\right)  -\left(
\ref{ObserverODE12}\right)  $ then $\left(  \widehat{\theta},\widehat
{v}\right)  $ is valued in $\left[  0,1\right]  \times\left[  0,v_{\max
}\right]  $. Let $\forall t\geq0$, $\theta_{1}=\max\left\{  0,-\widehat
{\theta}\right\}  $, $\theta_{2}=\max\left\{  0,\widehat{\theta}-1\right\}
,v_{1}=\max\left\{  0,-\widehat{v}\right\}  $ and $v_{2}=\max\left\{
0,\widehat{v}-v_{\max}\right\}  $. We have $\theta_{i}\left(  0\right)
=v_{i}\left(  0\right)  =0$, $\forall i\in\left\{  1,2\right\}  $. Without
loss of generality we can assume that the functions $\theta_{i}$ and $v_{i}$
($i\in\left\{  1,2\right\}  $) are positive on an open interval $\left]
0,T\right[  $. Then we have
\begin{align*}
\theta_{1}\left(  t\right)   &  =-\int\nolimits_{0}^{t}\alpha\left(  s\right)
\left(  1+w\left(  s\right)  \theta_{1}\left(  s\right)  \right)
ds-\int\nolimits_{0}^{t}k_{1}\left(  s\right)  \phi_{1}\left(  s,-\theta
_{1}\left(  s\right)  ,\widehat{v}\left(  s\right)  \right)  ds\\
&  -\int\nolimits_{0}^{t}k_{2}\left(  s\right)  \phi_{2}\left(  s,-\theta
_{1}\left(  s\right)  \right)  ds\\
&  \leq0
\end{align*}
and%
\begin{align*}
\theta_{2}\left(  t\right)   &  =\int\nolimits_{0}^{t}\alpha\left(  s\right)
\left(  1-w\left(  s\right)  -w\left(  s\right)  \theta_{2}\left(  s\right)
\right)  ds+\int\nolimits_{0}^{t}k_{1}\left(  s\right)  \phi_{1}\left(
s,1+\theta_{2}\left(  s\right)  ,\widehat{v}\left(  s\right)  \right)  ds\\
&  +\int\nolimits_{0}^{t}k_{2}\left(  s\right)  \phi_{2}\left(  s,1+\theta
_{1}\left(  s\right)  \right)  ds\\
&  \leq0\text{.}%
\end{align*}
That is $\theta_{1}$ and $\theta_{2}$ are identically null and $\widehat
{\theta}$ is $\left[  0,1\right]  $-valued. In the same manner one can show
that $\widehat{v}$ is valued in $\left[  0,v_{\max}\right]  $.

Now let show the existence and the uniqueness of the solution of $\left(
\ref{ObserverODE11}\right)  -\left(  \ref{ObserverODE12}\right)  $. It
suffices to establish existence and uniqueness of a local solution and use the
fact every local solution of $\left(  \ref{ObserverODE11}\right)  -\left(
\ref{ObserverODE12}\right)  $ is bounded to conclude using Theorem 5.7 in
\cite{benzoni}.

Let consider the function $F:\left(  t,x,y\right)  \in%
%TCIMACRO{\U{211d} }%
%BeginExpansion
\mathbb{R}
%EndExpansion
_{+}\times\left[  0,1\right]  \times\left[  0,v_{\max}\right]  \mapsto%
%TCIMACRO{\U{211d} }%
%BeginExpansion
\mathbb{R}
%EndExpansion
^{2}$ defined by
\[
F^{1}\left(  t,x,y,z\right)  =\alpha\left(  t\right)  \left(  1-xw\left(
t\right)  \right)  +k_{1}\phi_{1}\left(  t,x,y\right)  +k_{2}\phi_{2}\left(
t,x\right)  \text{,}%
\]
and%
\[
F^{2}\left(  t,x,y\right)  =\beta\left(  t,x\right)  \phi_{3}\left(
t,x,y\right)  \text{.}%
\]
The function $F$ is measurable with respect to the time $t$ and continuous
with respect to $\left(  x,y\right)  $. Using the Carath\'{e}odory theorem
there is a local solution of $\left(  \ref{ObserverODE11}\right)  -\left(
\ref{ObserverODE12}\right)  $. Moreover, $F$ is Lipschitz continuous with
respect to $\left(  x,y,z\right)  $. Therefore, the solution is unique and
global. Note that the hypothesis \textbf{(}\ref{Hypothesis_7}\textbf{)}
implies that $\forall t\geq0,$ the function $\beta\left(  t,.\right)  $ is
Lipschitz continuous and bounded on the set $\left[  0,1\right]  $.
\end{proof}

Now we state the main results of this subsection. Let define the function
$\delta:%
%TCIMACRO{\U{211d} }%
%BeginExpansion
\mathbb{R}
%EndExpansion
\rightarrow\left\{  0,1\right\}  $ such as
\[
\delta\left(  x\right)  =\left\{
\begin{array}
[c]{l}%
1\text{, }x\in\left]  0,1\right[ \\
0\text{, otherwise}%
\end{array}
\right.  \text{.}%
\]

\begin{theorem}
\label{ObsLocODE}Let consider the system $\left(  \ref{ObserverODE11}\right)
-\left(  \ref{ObserverODE12}\right)  $ and assume that $\widehat{v}\left(
0\right)  =v\left(  0\right)  $.

\begin{enumerate}
\item[$\left(  i\right)  $] If the functions $k_{1}$ and $k_{2}$ are
identically null and
\begin{equation}
\inf\left\{  \alpha\left(  t\right)  ;t>0\right\}  >0
\end{equation}
then $\widehat{\theta}$ is a globally exponentially asymptotically stable
observer for $\theta$.

\item[$\left(  ii\right)  $] If the function $k_{1}$ is identically null and
$k_{2}$ is not identically null then $\widehat{\theta}$ is a a locally
asymptotically stable observers for $\theta$. Moreover, if there is a positive
function $C_{\gamma}\in L_{loc}^{\infty}\left(
%TCIMACRO{\U{211d} }%
%BeginExpansion
\mathbb{R}
%EndExpansion
_{+};%
%TCIMACRO{\U{211d} }%
%BeginExpansion
\mathbb{R}
%EndExpansion
_{+}\right)  $ such that
\begin{equation}
\left\vert \overline{\gamma}\left(  t,y_{1},y_{2},y_{3}\right)  -\overline
{\gamma}\left(  t,z_{1},z_{2},z_{3}\right)  \right\vert \geq C_{\gamma}\left(
t\right)  \left\Vert y-z\right\Vert \label{ConditionCoercive}%
\end{equation}
then $\widehat{\theta}$ is a a globally exponentially asymptotically stable
observers for $\theta$.

\item[$\left(  iii\right)  $] If the function $k_{2}$ is identically null and
$\forall t>0$,
\begin{equation}
\alpha w+k_{1}\delta\left(  \widehat{\theta}\right)  +\frac{k_{1}\delta\left(
\widehat{\theta}\right)  \beta\left(  t,\theta\right)  \left(  \eta v_{\max
}\left(  1+\varepsilon-\theta\right)  -v\right)  }{\alpha\eta vv_{\max}\left(
1-\theta w\right)  }>0\text{.} \label{ConditionStabilityOde1}%
\end{equation}
then $\widehat{\theta}$ is a a locally exponentially asymptotically stable
observers for $\theta$.

\item[$\left(  iv\right)  $] If the functions $k_{1}$ and $k_{2}$ are not
identically null, and $\forall t>0$,
\begin{equation}
\alpha w+k_{1}\delta\left(  \widehat{\theta}\right)  +k_{2}\left(  t\right)
\phi_{2}\left(  t,\widehat{\theta}\right)  +\frac{k_{1}\delta\left(
\widehat{\theta}\right)  \beta\left(  t,\theta\right)  \left(  \eta v_{\max
}\left(  1+\varepsilon-\theta\right)  -v\right)  }{\alpha\eta vv_{\max}\left(
1-w\theta\right)  }>0 \label{ConditionStabilityOde2}%
\end{equation}
then $\widehat{\theta}$ is a a locally exponentially asymptotically stable
observers for $\theta$. Moreover, if the condition $\left(
\ref{ConditionCoercive}\right)  $ is satisfied and
\begin{equation}
k_{2}\left\vert \phi_{2}\left(  t,\widehat{\theta}\right)  \right\vert
>k_{1}\phi_{1}\left(  t,\widehat{\theta},\widehat{v}\right)
\label{ConditionDominance}%
\end{equation}
then $\widehat{\theta}$ is a a globally exponentially asymptotically stable
observers for $\theta$.
\end{enumerate}
\end{theorem}

\begin{proof}
Let $\widehat{e}=\theta-\widehat{\theta}$ be the error of estimation.

\begin{enumerate}
\item[$\left(  i\right)  $] The error $\widehat{e}$ satisfies
\begin{equation}
d_{t}\widehat{e}=-\alpha\left(  t\right)  w\left(  t\right)  \widehat{e}
\label{EqErrorNat}%
\end{equation}
and
\[
\widehat{e}\left(  t\right)  =\exp\left(  -\int\nolimits_{0}^{t}\alpha\left(
s\right)  w\left(  s\right)  ds\right)  \widehat{e}\left(  0\right)  \text{.}%
\]
The result follows.

\item[$\left(  ii\right)  $] The error $\widehat{e}$ satisfies
\begin{align}
d_{t}\widehat{e}^{2}  &  =-2\alpha\left(  t\right)  w\left(  t\right)
\widehat{e}^{2}\left(  t\right)  -2k_{2}\left(  t\right)  \phi_{2}\left(
t,\widehat{\theta}\left(  t\right)  \right)  \widehat{e}\left(  t\right) \\
&  =-\alpha\left(  t\right)  w\left(  t\right)  \widehat{e}^{2}\left(
t\right)  -2k_{2}\left(  t\right)  \left(  1-\rho\left(  t\right)  \right)
\left(  \overline{\gamma}\left(  t,\theta\left(  t\right)  ,v\left(  t\right)
,\rho\left(  t\right)  \right)  -\overline{\gamma}\left(  t,\widehat{\theta
}\left(  t\right)  ,v\left(  t\right)  ,\rho\left(  t\right)  \right)
\right)  \widehat{e}\left(  t\right) \nonumber\\
&  \leq-\alpha\left(  t\right)  w\left(  t\right)  \widehat{e}^{2}\left(
t\right) \nonumber
\end{align}
and
\[
\widehat{e}^{2}\left(  t\right)  \leq\exp\left(  -\int\nolimits_{0}^{t}%
\alpha\left(  s\right)  w\left(  s\right)  ds\right)  \widehat{e}^{2}\left(
0\right)  \text{.}%
\]
Moreover, if the condition $\left(  \ref{ConditionCoercive}\right)  $ is
satisfied then
\begin{equation}
d_{t}\widehat{e}^{2}\leq-\alpha\left(  t\right)  w\left(  t\right)
\widehat{e}^{2}\left(  t\right)  -2k_{2}\left(  t\right)  \left(
1-\rho\left(  t\right)  \right)  C_{\gamma}\left(  t\right)  \widehat{e}%
^{2}\left(  t\right) \nonumber
\end{equation}
and
\[
\widehat{e}^{2}\left(  t\right)  \leq\exp\left(  -\int\nolimits_{0}^{t}\left(
\alpha\left(  s\right)  w\left(  s\right)  +2k_{2}\left(  s\right)  \left(
1-\rho\left(  s\right)  \right)  C_{\gamma}\left(  s\right)  \right)
ds\right)  \widehat{e}^{2}\left(  0\right)  \text{.}%
\]

\item[$\left(  iii\right)  $] We have%
\begin{align*}
\partial_{\theta}\frac{\beta\left(  s,\theta\right)  }{\beta\left(
s,\theta\right)  -d_{t}v}  &  =\frac{d_{t}v\partial_{\theta}\beta\left(
s,\theta\right)  -\beta\left(  s,\theta\right)  \partial_{\theta}d_{t}%
v}{\left(  \beta\left(  s,\theta\right)  -d_{t}v\right)  ^{2}}\\
&  =\frac{\beta^{2}\left(  s,\theta\right)  \left(  \left(  1+\varepsilon
-\theta\right)  \partial_{\theta}v+v\right)  }{\eta v_{\max}\left(
1+\varepsilon-\theta\right)  ^{2}\left(  \beta\left(  s,\theta\right)
-d_{t}v\right)  ^{2}}%
\end{align*}
and at the neighborhood of $0$,
\begin{align*}
\frac{\beta\left(  s,\widehat{\theta}\right)  }{\beta\left(  s,\widehat
{\theta}\right)  -d_{t}\widehat{v}}  &  =\frac{\beta\left(  s,\theta\right)
}{\left(  \beta\left(  s,\theta\right)  -d_{t}v\right)  }-\frac{\beta
^{2}\left(  s,\theta\right)  \left(  \left(  1+\varepsilon-\theta\right)
\partial_{\theta}v+v\right)  }{\eta v_{\max}\left(  1+\varepsilon
-\theta\right)  ^{2}\left(  \beta\left(  s,\theta\right)  -d_{t}v\right)
^{2}}\widehat{e}+\widehat{e}\mathcal{O}\left(  \widehat{e}\right) \\
&  =\frac{\eta v_{\max}\left(  1+\varepsilon-\theta\right)  }{v}-\frac{\eta
v_{\max}}{v}\widehat{e}-\frac{\eta v_{\max}\left(  1+\varepsilon
-\theta\right)  \partial_{\theta}v}{v^{2}}\widehat{e}+\widehat{e}%
\mathcal{O}\left(  \widehat{e}\right)
\end{align*}
It follows that
\begin{align*}
\frac{v}{\widehat{v}}  &  =\frac{v}{\eta v_{\max}\left(  1+\varepsilon
-\theta\right)  }\frac{\beta\left(  s,\widehat{\theta}\right)  }{\beta\left(
s,\widehat{\theta}\right)  -d_{t}\widehat{v}}\\
&  =1-\left(  \frac{1}{1+\varepsilon-\theta}+\frac{\partial_{\theta}v}%
{v}\right)  \widehat{e}-\frac{v\widehat{e}}{\eta v_{\max}\left(
1+\varepsilon-\theta\right)  }\mathcal{O}\left(  \widehat{e}\right)  \text{.}%
\end{align*}
finally,
\begin{align*}
d_{t}\widehat{e}  &  =-\alpha\left(  t\right)  w\left(  t\right)  \widehat
{e}-k_{1}\phi_{1}\left(  t,\widehat{\theta},\widehat{v}\right) \\
&  =-\frac{k_{1}\left(  t\right)  \delta\left(  \widehat{\theta}\right)
\left(  1+\varepsilon-\theta\right)  \partial_{\theta}v}{v}\widehat{e}-\left(
\alpha\left(  t\right)  w\left(  t\right)  -k_{1}\left(  t\right)
\delta\left(  \widehat{\theta}\right)  \right)  \widehat{e}\\
&  \ \ \ \ -\frac{k_{1}\left(  t\right)  \delta\left(  \widehat{\theta
}\right)  v}{\eta v_{\max}}\widehat{e}\mathcal{O}\left(  \widehat{e}\right) \\
&  \approx-\frac{k_{1}\left(  t\right)  \delta\left(  \widehat{\theta}\right)
\beta\left(  t,\theta\right)  \left(  \eta v_{\max}\left(  1+\varepsilon
-\theta\right)  -v\right)  }{\alpha\eta vv_{\max}\left(  1-\theta w\right)
}\widehat{e}-\left(  \alpha\left(  t\right)  w\left(  t\right)  +k_{1}\left(
t\right)  \delta\left(  \widehat{\theta}\right)  \right)  \widehat{e}%
\end{align*}
The given approximmation on the dynamics of $\widehat{e}$ shows clearly that
$\widehat{\theta}$ is a locally exponentially asymptotically stable observers
for $\theta$ if the condition $\left(  \ref{ConditionStabilityOde1}\right)  $
is satified.

\item[$\left(  iv\right)  $] With the same arguments developped in $\left(
ii\right)  $ and $\left(  iii\right)  $ we easily get that $\widehat{\theta}$
is a locally exponentially asymptotically stable observers for $\theta$. Using
the conditions $\left(  \ref{ConditionCoercive}\right)  $ and $\left(
\ref{ConditionDominance}\right)  $ we are sure that even the behaviour of the
term with $\phi_{1}$ tends to introduce unstability, the stability is
maintained by the term with $\phi_{2}$.
\end{enumerate}
\end{proof}

\subsection{Numerical evaluation of the
observers\label{ObserversEvalutionNum1}}

This section is devoted to numerical simulations in order to check the
effectiveness of the observers we studied theoretically in the subsection
\ref{SubsectionDesignTheoric1}. The Theorem \ref{ObsLocODE} guarantees
asymptotical stability of the observers, but pratically the convergence is
needed in a finite time. Since the coffee cultivation campaign is a cyle of
one year, our simulations will cover the year. Parameters are taken following
\cite{fotsa}. The control strategy $u$ is taken with many variations in order
to study their impact on the.performances of the observer. For every time
$t\geq0$, $u$ is given by
\begin{equation}
u(t)=\sin^{2}\left(  \omega_{1}\left(  t-\varphi_{1}\right)  ^{2}\right)
\exp\left(  -\omega_{2}\left(  t-\varphi_{2}\right)  ^{2}\right)  \text{.}%
\end{equation}
The functions $\alpha,\beta$ and $\gamma$ are taken with the following form
\begin{equation}
\alpha\left(  t\right)  =p_{1}\left(  t\right)  +b_{1}\left(  1-\cos\left(
c_{1}t\right)  \right)  \left(  t-d_{1}\right)  ^{2}\text{, }\forall t\in%
%TCIMACRO{\U{211d} }%
%BeginExpansion
\mathbb{R}
%EndExpansion
_{+}\text{,}%
\end{equation}%
\begin{equation}
\beta\left(  t,x\right)  =b_{2}\left(  1-\cos\left(  c_{2}t\right)  \right)
\left(  t-d_{2}\right)  ^{2}p_{2}\left(  x\right)  \text{, }\forall\left(
t,x\right)  \in%
%TCIMACRO{\U{211d} }%
%BeginExpansion
\mathbb{R}
%EndExpansion
_{+}\times%
%TCIMACRO{\U{211d} }%
%BeginExpansion
\mathbb{R}
%EndExpansion
\text{,}%
\end{equation}
and
\begin{equation}
\gamma\left(  t,x_{1},x_{2},x_{3}\right)  =b_{3}\left(  1-\cos\left(
c_{3}t\right)  \right)  \left(  t-d_{3}\right)  ^{2}\left(  x_{1}-\kappa
x_{3}\right)  x_{2}\text{, }\forall\left(  t,x\right)  \in%
%TCIMACRO{\U{211d} }%
%BeginExpansion
\mathbb{R}
%EndExpansion
_{+}\times%
%TCIMACRO{\U{211d} }%
%BeginExpansion
\mathbb{R}
%EndExpansion
^{3}\text{.}%
\end{equation}
$p_{1}$ is a nonnegative function of the time $t$ and $p_{2}$ is a real
positive function of $x.$ $\forall i\in\left\{  1,2,3\right\}  ,$ $b_{i}$,
$c_{i}$, and $d_{i}$ are positive coefficients corresponding respectively to
the maximal amplitude, the pulsation and the global maximun of $\alpha,\beta$
and $\gamma$. $\kappa$ is a positive constant regulating the evolution of the
rot volume with respect to the inhibition rate. The terms $1-cos(\left(
c_{i}t\right)  $ represent the seasonality probably due to climatic and
environmental variations.

Theoretically the functions $k_{1}$ and $k_{2}$ might take abitrary positive
numbers and we can set them constant without loss of generality. However, for
the stability of the numerical scheme we were constrained to set them less
than $1/10\Delta t$. Here $\Delta t$ denotes the constant time step. The
initial conditions are taken such as $\theta\left(  0\right)  \in\left\{
0.05,0.25,0.50,0.75\right\}  $, $v\left(  0\right)  \in\left\{
0,0.25,0.50,0.75\right\}  $, $\rho\left(  0\right)  \in\left\{
0.25,0.50,0.75\right\}  ,$ $\widehat{\theta}\left(  0\right)  =0$,
$\widehat{v}\left(  0\right)  =v\left(  0\right)  $. Indeed, at the beginning
the fruits are small and the inhibition rate is relatively small.

The following table gives the assumed parameters values.

\begin{table}[ptbh]
\centering%
\begin{tabular}
[c]{|c|c|c|c|c|c|}\hline
\textbf{Parameters} & \textbf{Values} & \textbf{Source} & \textbf{Parameters}
& \textbf{Values} & \textbf{Source}\\\hline
$b_{1}$ & $5\ln\left(  10\right)  $ & Assumed & $p_{1}\left(  t\right)  $ &
$0$ & Assumed\\
$b_{2}$ & $v_{\max}\ln\left(  10^{5}v_{\max}\left(  1-\varepsilon\eta^{\ast
}\right)  \right)  /2$ & Assumed & $p_{2}\left(  x\right)  $ & $2-x$ or
$\left(  2-x\right)  ^{2}$ & Assumed\\
$b_{3}$ & $v_{\max}\ln(10^{5}v_{\max})$ & Assumed & $\kappa$ & $1$ & Assumed\\
$d_{1}$ & $7.5\times10^{-1}$ & \cite{fotsa} & $v_{\max}$ & $1$ $cm^{2}$ &
Assumed\\
$d_{2}$ & $7.5\times10^{-1}$ & \cite{fotsa} & $\varepsilon$ & $10^{-4}$ &
Assumed\\
$d_{3}$ & $7.5\times10^{-1}$ & \cite{fotsa} & $\eta\left(  t\right)  $ &
$1/\left(  1+\varepsilon\right)  $ $cm^{2}$ & Assumed\\
$c_{1}$ & $10\pi$ & \cite{fotsa} & $\sigma$ & $0.9$ & \cite{fotsa}\\
$c_{2}$ & $10\pi$ & \cite{fotsa} & $k_{1}$ & $0$ or $10^{3}$ & Assumed\\
$c_{3}$ & $10\pi$ & \cite{fotsa} & $k_{2}$ & $0$ or $10^{3}$ & Assumed\\
$\omega_{1}$ & $25\pi$ & Assumed & $\Delta t$ & $10^{-4}$ & Assumed\\
$\omega_{2}$ & $10$ & Assumed & $\varphi_{1}$ & $0.6$ & Assumed\\
$\varphi_{1}$ & $0.4$ & Assumed &  &  & \\\hline
\end{tabular}
%title of Table
\caption{Simulation parameters for the non spatial model}%
\label{TableParam1}%
\end{table}

In the remainder of the subsection we present simulations of the observers and
their estimation relative errors dynamics. The relative error seems to be a
good mean to evaluate the performance. For each scenario we display two groups
figures. The first one represents the dynamics both of the inhibition rate and
observers corresponding to each values of $\rho\left(  0\right)  \leq
\theta\left(  0\right)  $. The second group of figures shows relatives errors
of observers corresponding to each values of $\rho\left(  0\right)  \leq
\theta\left(  0\right)  $.\newpage

\subsubsection{Simulations with $\theta\left(  0\right)  =v\left(  0\right)
=0.05$}

The figures \ref{figInh11} and \ref{figErr11} show the behaviour of the
observers when the observation is started early. As we can see the global
performance of each observer is relatively good but the case $k_{1}=0$ and
$k_{2}=10^{3}$ seems to be the better one.

\begin{figure}[ptbh]
\centering%
\begin{tabular}
[c]{cc}%
\raisebox{-0cm}{\includegraphics[
natheight=12.171500cm,
natwidth=16.138599cm,
height=6.1549cm,
width=8.1385cm
]{./EstimInhRat0011.jpg}} & \raisebox{-0cm}{\includegraphics[
natheight=12.171500cm,
natwidth=16.138599cm,
height=6.1549cm,
width=8.1385cm
]{./EstimInhRat0111.jpg}}\\
$k_{1}=k_{2}=0$ & $k_{1}=0$ and $k_{2}=10^{3}$\\
\raisebox{-0cm}{\includegraphics[
natheight=12.171500cm,
natwidth=16.138599cm,
height=6.1549cm,
width=8.1385cm
]{./EstimInhRat1011.jpg}} & \raisebox{-0cm}{\includegraphics[
natheight=12.171500cm,
natwidth=16.138599cm,
height=6.1549cm,
width=8.1385cm
]{./EstimInhRat1111.jpg}}\\
$k_{1}=10^{3}$ and $k_{2}=0$ & $k_{1}=10^{3}$ and $k_{2}=10^{3}$%
\end{tabular}
%title of Figure
\caption{Inhibition rate and observers for $\theta\left(  0\right)  =v\left(
0\right)  =0.05$.}%
\label{figInh11}%
\end{figure}

\begin{figure}[ptbh]
\centering%
\begin{tabular}
[c]{cc}%
\raisebox{-0cm}{\includegraphics[
natheight=12.171500cm,
natwidth=16.138599cm,
height=6.1527cm,
width=8.1385cm
]{./ReAbsErr0011.jpg}} & \raisebox{-0cm}{\includegraphics[
natheight=12.171500cm,
natwidth=16.138599cm,
height=6.1527cm,
width=8.1385cm
]{./ReAbsErr0111.jpg}}\\
$k_{1}=k_{2}=0$ & $k_{1}=0$ and $k_{2}=10^{3}$\\
\raisebox{-0cm}{\includegraphics[
natheight=12.171500cm,
natwidth=16.138599cm,
height=6.1527cm,
width=8.1385cm
]{./ReAbsErr1011.jpg}} & \raisebox{-0cm}{\includegraphics[
natheight=12.171500cm,
natwidth=16.138599cm,
height=6.1527cm,
width=8.1385cm
]{./ReAbsErr1111.jpg}}\\
$k_{1}=10^{3}$ and $k_{2}=0$ & $k_{1}=10^{3}$ and $k_{2}=10^{3}$%
\end{tabular}
%title of Figure
\caption{Relative absolute estimation error for $\theta\left(  0\right)
=v\left(  0\right)  =0.05$.}%
\label{figErr11}%
\end{figure}\newpage

\subsubsection{Simulations with $\theta\left(  0\right)  =0.05$ and $v\left(
0\right)  =0.5$}

The figures \ref{figInh13} and \ref{figErr13} correspond to an observation
starting at the beginning of the disease but with a well developped berry. We
have better performances when $k_{2}=10^{3}$. With $k_{1}=10^{3}$ and
$k_{2}=0$ we notice an unstability behaviour in figure \ref{figErr13} although
the error takes the null value. That may be due to the constant sign of the
function $\phi_{1}$.

\begin{figure}[ptbh]
\centering%
\begin{tabular}
[c]{cc}%
\raisebox{-0cm}{\includegraphics[
natheight=12.171500cm,
natwidth=16.138599cm,
height=6.1549cm,
width=8.1385cm
]{./EstimInhRat0013.jpg}} & \raisebox{-0cm}{\includegraphics[
natheight=12.171500cm,
natwidth=16.138599cm,
height=6.1549cm,
width=8.1385cm
]{./EstimInhRat0113.jpg}}\\
$k_{1}=k_{2}=0$ & $k_{1}=0$ and $k_{2}=10^{3}$\\
\raisebox{-0cm}{\includegraphics[
natheight=12.171500cm,
natwidth=16.138599cm,
height=6.1549cm,
width=8.1385cm
]{./EstimInhRat1013.jpg}} & \raisebox{-0cm}{\includegraphics[
natheight=12.171500cm,
natwidth=16.138599cm,
height=6.1549cm,
width=8.1385cm
]{./EstimInhRat1113.jpg}}\\
$k_{1}=10^{3}$ and $k_{2}=0$ & $k_{1}=10^{3}$ and $k_{2}=10^{3}$%
\end{tabular}
%title of Figure
\caption{Inhibition rate and observers for $\theta\left(  0\right)  =0.05$ and
$v\left(  0\right)  =0.5$.}%
\label{figInh13}%
\end{figure}

\begin{figure}[ptbh]
\centering%
\begin{tabular}
[c]{cc}%
\raisebox{-0cm}{\includegraphics[
natheight=12.171500cm,
natwidth=16.138599cm,
height=6.1549cm,
width=8.1385cm
]{./ReAbsErr0013.jpg}} & \raisebox{-0cm}{\includegraphics[
natheight=12.171500cm,
natwidth=16.138599cm,
height=6.1549cm,
width=8.1385cm
]{./ReAbsErr0113.jpg}}\\
$k_{1}=k_{2}=0$ & $k_{1}=0$ and $k_{2}=10^{3}$\\
\raisebox{-0cm}{\includegraphics[
natheight=12.171500cm,
natwidth=16.138599cm,
height=6.1549cm,
width=8.1385cm
]{./ReAbsErr1013.jpg}} & \raisebox{-0cm}{\includegraphics[
natheight=12.171500cm,
natwidth=16.138599cm,
height=6.1549cm,
width=8.1385cm
]{./ReAbsErr1113.jpg}}\\
$k_{1}=10^{3}$ and $k_{2}=0$ & $k_{1}=10^{3}$ and $k_{2}=10^{3}$%
\end{tabular}
%title of Figure
\caption{Relative absolute estimation error for $\theta\left(  0\right)
=0.05$ and $v\left(  0\right)  =0.5$.}%
\label{figErr13}%
\end{figure}\newpage

\subsubsection{Simulations with $\theta\left(  0\right)  =0.75$ and $v\left(
0\right)  =0.05$}

In the figures \ref{figInh41} and \ref{figErr41} the inhibition rate is
already high while the berry is little. As expected the worst estimation
corresponds to $k_{1}=k_{2}=0$. The best estimations are still made when
$k_{2}=10^{3}$. However, when $k_{1}=10^{3}$, the same unstability behaviour
in figure \ref{figErr13} appears. Note that the performances of the observers
are better with $\rho\left(  0\right)  $ small.

\begin{figure}[ptbh]
\centering%
\begin{tabular}
[c]{cc}%
\raisebox{-0cm}{\includegraphics[
natheight=12.171500cm,
natwidth=16.138599cm,
height=6.1549cm,
width=8.1385cm
]{./EstimInhRat0041.jpg}} & \raisebox{-0cm}{\includegraphics[
natheight=12.171500cm,
natwidth=16.138599cm,
height=6.1549cm,
width=8.1385cm
]{./EstimInhRat0141.jpg}}\\
$k_{1}=k_{2}=0$ & $k_{1}=0$ and $k_{2}=10^{3}$\\
\raisebox{-0cm}{\includegraphics[
natheight=12.171500cm,
natwidth=16.138599cm,
height=6.1549cm,
width=8.1385cm
]{./EstimInhRat1041.jpg}} & \raisebox{-0cm}{\includegraphics[
natheight=12.171500cm,
natwidth=16.138599cm,
height=6.1549cm,
width=8.1385cm
]{./EstimInhRat1141.jpg}}\\
$k_{1}=10^{3}$ and $k_{2}=0$ & $k_{1}=10^{3}$ and $k_{2}=10^{3}$%
\end{tabular}
%title of figure
\caption{Inhibition rate and observers for $\theta\left(  0\right)  =0.75$ and
$v\left(  0\right)  =0.05$.}%
\label{figInh41}%
\end{figure}

\begin{figure}[ptbh]
\centering%
\begin{tabular}
[c]{cc}%
\raisebox{-0cm}{\includegraphics[
natheight=12.171500cm,
natwidth=16.138599cm,
height=6.1527cm,
width=8.1385cm
]{./ReAbsErr0041.jpg}} & \raisebox{-0cm}{\includegraphics[
natheight=12.171500cm,
natwidth=16.138599cm,
height=6.1527cm,
width=8.1385cm
]{./ReAbsErr0141.jpg}}\\
$k_{1}=k_{2}=0$ & $k_{1}=0$ and $k_{2}=10^{3}$\\
\raisebox{-0cm}{\includegraphics[
natheight=12.171500cm,
natwidth=16.138599cm,
height=6.1527cm,
width=8.1385cm
]{./ReAbsErr1041.jpg}} & \raisebox{-0cm}{\includegraphics[
natheight=12.171500cm,
natwidth=16.138599cm,
height=6.1527cm,
width=8.1385cm
]{./ReAbsErr1141.jpg}}\\
$k_{1}=10^{3}$ and $k_{2}=0$ & $k_{1}=10^{3}$ and $k_{2}=10^{3}$%
\end{tabular}
%title of Table
\caption{Relative absolute estimation error for $\theta\left(  0\right)
=0.75$ and $v\left(  0\right)  =0.05$. }%
\label{figErr41}%
\end{figure}\newpage

\subsubsection{Simulations with $\theta\left(  0\right)  =0.75$ and $v\left(
0\right)  =0.5$}

For the figures \ref{figInh43} and \ref{figErr43} the inhibition rate is
already high and the fruit is relatively mature when the observation starts.
The worst case are still when $k_{2}=0$ and the best are with $k_{2}=10^{3}$.
Again the performances of the observers are better with $\rho\left(  0\right)
$ small. It seems that starting observation with $\theta\left(  0\right)  $
small and $v\left(  0\right)  $ big guarantees better results.

\begin{figure}[ptbh]
\centering%
\begin{tabular}
[c]{cc}%
\raisebox{-0cm}{\includegraphics[
natheight=12.171500cm,
natwidth=16.138599cm,
height=6.1549cm,
width=8.1385cm
]{./EstimInhRat0043.jpg}} & \raisebox{-0cm}{\includegraphics[
natheight=12.171500cm,
natwidth=16.138599cm,
height=6.1549cm,
width=8.1385cm
]{./EstimInhRat0143.jpg}}\\
$k_{1}=k_{2}=0$ & $k_{1}=0$ and $k_{2}=10^{3}$\\
\raisebox{-0cm}{\includegraphics[
natheight=12.171500cm,
natwidth=16.138599cm,
height=6.1549cm,
width=8.1385cm
]{./EstimInhRat1043.jpg}} & \raisebox{-0cm}{\includegraphics[
natheight=12.171500cm,
natwidth=16.138599cm,
height=6.1549cm,
width=8.1385cm
]{./EstimInhRat1143.jpg}}\\
$k_{1}=10^{3}$ and $k_{2}=0$ & $k_{1}=10^{3}$ and $k_{2}=10^{3}$%
\end{tabular}
%title of figure
\caption{Inhibition rate and observers for $\theta\left(  0\right)  =0.75$ and
$v\left(  0\right)  =0.5$.}%
\label{figInh43}%
\end{figure}

\begin{figure}[ptbh]
\centering%
\begin{tabular}
[c]{cc}%
\raisebox{-0cm}{\includegraphics[
natheight=12.171500cm,
natwidth=16.138599cm,
height=6.1527cm,
width=8.1385cm
]{./ReAbsErr0043.jpg}} & \raisebox{-0cm}{\includegraphics[
natheight=12.171500cm,
natwidth=16.138599cm,
height=6.1527cm,
width=8.1385cm
]{./ReAbsErr0143.jpg}}\\
$k_{1}=k_{2}=0$ & $k_{1}=0$ and $k_{2}=10^{3}$\\
\raisebox{-0cm}{\includegraphics[
natheight=12.171500cm,
natwidth=16.138599cm,
height=6.1527cm,
width=8.1385cm
]{./ReAbsErr1043.jpg}} & \raisebox{-0cm}{\includegraphics[
natheight=12.171500cm,
natwidth=16.138599cm,
height=6.1527cm,
width=8.1385cm
]{./ReAbsErr1143.jpg}}\\
$k_{1}=10^{3}$ and $k_{2}=0$ & $k_{1}=10^{3}$ and $k_{2}=10^{3}$%
\end{tabular}
%title of figure
\caption{Relative absolute estimation error for $\theta\left(  0\right)
=0.75$ and $v\left(  0\right)  =0.5$.}%
\label{figErr43}%
\end{figure}\newpage

\section{ Observation for the spatial model\label{SectionObserverPDE}}

\qquad The aim of this section is similar to the previous section. We design
observers for the spatial dynamical system $\left(  \ref{ModelSpace1}\right)
-\left(  \ref{ModelSpace3}\right)  ,$ $\left(  \ref{ModelSpace6}\right)  ,$
$\left(  \ref{ModelSpace42}\right)  $ and $\left(  \ref{ModelSpace52}\right)
$ in order to estimate the inhibition rate $\theta$ and we improve them using
computer simulations.

\subsection{Theoretical design of the spatial model
observers\label{SubsectionDesignTheoric2}}

Let consider the following systems : $\forall t\geq0,\forall x\in\overline{U}%
$,
\begin{align}
\partial_{t}\widehat{\theta}  &  =\alpha\left(  t,x\right)  \left(  1-w\left(
t,x\right)  \widehat{\theta}\right)  +K_{1}\left(  t,x\right)  \Phi_{1}\left(
t,x,\widehat{\theta},\widehat{v}\right) \label{ObserverPDE11}\\
&  +K_{2}\left(  t,x\right)  \Phi_{2}\left(  t,x,\widehat{\theta}\right)
+\operatorname{div}\left(  A\left(  t,x\right)  \nabla\widehat{\theta}\right)
,\text{ on }%
%TCIMACRO{\U{211d} }%
%BeginExpansion
\mathbb{R}
%EndExpansion
_{+}^{\ast}\times U\nonumber
\end{align}%
\begin{equation}
\left\langle A\left(  t,x\right)  \nabla\widehat{\theta}_{1}\left(
t,x\right)  ,n\left(  x\right)  \right\rangle =0,\text{ on }%
%TCIMACRO{\U{211d} }%
%BeginExpansion
\mathbb{R}
%EndExpansion
_{+}^{\ast}\times\partial U \label{ObserverPDE12}%
\end{equation}%
\begin{equation}
\partial_{t}\widehat{v}=\beta\left(  t,x,\widehat{\theta}\right)  \Phi
_{3}\left(  t,x,\widehat{\theta},\widehat{v}\right)  \label{ObserverPDE13}%
\end{equation}

where
\[
\Phi_{1}\left(  t,x,y,z\right)  =\left\{
\begin{array}
[c]{l}%
\left(  1-v\left(  t,x\right)  /z\right)  \left(  1+\varepsilon-y\right)
\text{, if }v\left(  t,x\right)  \leq y\text{ and }\left(  x,y,z\right)
\in\overline{U}\times\left]  0,1\right[  \times%
%TCIMACRO{\U{211d} }%
%BeginExpansion
\mathbb{R}
%EndExpansion
_{+}^{\ast}\\
0\text{, otherwise}%
\end{array}
\right.  ,
\]%
\[
\Phi_{2}\left(  t,x,y\right)  =\left\{
\begin{array}
[c]{l}%
\partial_{t}\rho\left(  t,x\right)  -\overline{\gamma}\left(  t,y,v\left(
t,x\right)  ,\rho\left(  t,x\right)  \right)  \left(  1-\rho\left(
t,x\right)  \right)  \text{, if }\left(  x,y\right)  \in\overline{U}%
\times\left]  0,1\right[ \\
0\text{, otherwise}%
\end{array}
\right.
\]
and
\[
\Phi_{3}\left(  t,x,y,z\right)  =1-\frac{z}{\left(  1+\varepsilon-y\right)
\eta\left(  t,x\right)  v_{\max}}\text{.}%
\]

We make the following assumption:

\begin{claim}
\label{Hypothesis_16}$K_{1},K_{2}\in L_{loc}^{\infty}\left(
%TCIMACRO{\U{211d} }%
%BeginExpansion
\mathbb{R}
%EndExpansion
_{+};L^{\infty}\left(  U;%
%TCIMACRO{\U{211d} }%
%BeginExpansion
\mathbb{R}
%EndExpansion
_{+}\right)  \right)  $.
\end{claim}

By a solution of the system $\left(  \ref{ObserverPDE11}\right)  -\left(
\ref{ObserverPDE13}\right)  $ we mean an absolutely continuous function with
respect to the time $\left(  \widehat{\theta},\widehat{v}\right)  $ which
satisfies $\forall t>0,\forall x\in U,$
\begin{align*}
\widehat{\theta}\left(  t,x\right)   &  =\widehat{\theta}\left(  0,x\right)
+\int\nolimits_{0}^{t}\alpha\left(  s,x\right)  \left(  1-w\left(  s,x\right)
\widehat{\theta}\left(  s,x\right)  \right)  ds+\operatorname{div}\left(
A\left(  s,x\right)  \nabla\widehat{\theta}_{1}\right)  ds\\
&  +\int\nolimits_{0}^{t}K_{1}\left(  s,x\right)  \Phi_{1}\left(
s,x,\widehat{\theta},\widehat{v}\right)  +\int\nolimits_{0}^{t}K_{2}\left(
s,x\right)  \Phi_{2}\left(  s,x,\widehat{\theta}\right)
\end{align*}%
\[
\widehat{v}\left(  t,x\right)  =\widehat{v}\left(  0,x\right)  +\int
\nolimits_{0}^{t}\beta\left(  s,x,\widehat{\theta}\right)  \Phi_{3}\left(
s,x,\widehat{\theta},\widehat{v}\right)  ds
\]

We have the following

\begin{proposition}
\label{ExistenceUniciteObp1}If $\left(  \widehat{\theta}\left(  0,.\right)
,\widehat{v}\left(  0,.\right)  \right)  \in L^{2}\left(  U;\left[
0,1\right[  \times\left[  0,v_{\max}\right]  \right)  $ then the system
$\left(  \ref{ObserverPDE11}\right)  -\left(  \ref{ObserverPDE13}\right)  $
has a unique weak solution $\left(  \widehat{\theta},\widehat{v}\right)  \in
C\left(
%TCIMACRO{\U{211d} }%
%BeginExpansion
\mathbb{R}
%EndExpansion
_{+};H^{1}\left(  U;\left[  0,1\right]  \right)  \right)  \times C\left(
%TCIMACRO{\U{211d} }%
%BeginExpansion
\mathbb{R}
%EndExpansion
_{+};L^{2}\left(  U;\left[  0,v_{\max}\right]  \right)  \right)  $.
\end{proposition}

\begin{proof}
We first show that if $\left(  \widehat{\theta},\widehat{v}\right)  $ is a
local solution of the system $\left(  \ref{ObserverPDE11}\right)  -\left(
\ref{ObserverPDE13}\right)  $ then $\forall t>0,$ $\left(  \widehat{\theta
}\left(  t,.\right)  ,\widehat{v}\left(  t,.\right)  \right)  $ is valued in
$\left[  0,1\right]  \times\left[  0,v_{\max}\right]  $. Let $\forall t\geq0,$
$\widehat{\theta}_{1}=\max\left\{  0,-\widehat{\theta}\right\}  ,$
$\widehat{\theta}_{2}=\max\left\{  0,\widehat{\theta}-1\right\}  ,$
$\widehat{v}_{1}=\max\left\{  0,-\widehat{v}\right\}  ,$ $\widehat{v}_{2}%
=\max\left\{  0,\widehat{v}-v_{\max}\right\}  $. We have $\widehat{\theta}%
_{1}\left(  0,.\right)  =\widehat{\theta}_{2}\left(  0,.\right)  =\widehat
{v}_{1}\left(  0,.\right)  =\widehat{v}_{2}\left(  0,.\right)  =0$. Let
$S_{i}^{\theta}$ and $S_{i}^{v}$ be convex open subsets of $%
%TCIMACRO{\U{211d} }%
%BeginExpansion
\mathbb{R}
%EndExpansion
_{+}$ where respectively functions $\widehat{\theta}_{i}\left(  t,.\right)  $
and $\widehat{v}_{i}\left(  t,.\right)  $ ($i\in\left\{  1,2\right\}  $) are
positive on subsets of $U$ with a positive measure. Then $\forall t\in
S_{i}^{\theta}$ (respectively $S_{i}^{v}$) we have%
\begin{align*}
d_{t}\left\Vert \widehat{\theta}_{1}\left(  t,.\right)  \right\Vert
_{L^{2}\left(  U;%
%TCIMACRO{\U{211d} }%
%BeginExpansion
\mathbb{R}
%EndExpansion
\right)  }^{2}/2  &  =\int\nolimits_{U}\widehat{\theta}_{1}\left(  t,x\right)
\partial_{t}f_{1}\left(  t,x\right)  dx\\
&  =-\int\nolimits_{U}\alpha\left(  t,x\right)  \left(  1+w\left(  t,x\right)
\widehat{\theta}_{1}\left(  t,x\right)  \right)  \widehat{\theta}_{1}\left(
t,x\right)  dx\\
&  -\int\nolimits_{U}K_{1}\left(  t,x\right)  \Phi_{1}\left(  t,x,-\widehat
{\theta}_{1}\left(  t,x\right)  ,\widehat{v}\right)  dx\\
&  -\int\nolimits_{U}K_{2}\left(  t,x\right)  \Phi_{2}\left(  t,x,-\widehat
{\theta}_{1}\left(  t,x\right)  \right)  dx\\
&  \leq0
\end{align*}%
\begin{align*}
d_{t}\left\Vert \widehat{\theta}_{2}\left(  t,.\right)  \right\Vert
_{L^{2}\left(  U;%
%TCIMACRO{\U{211d} }%
%BeginExpansion
\mathbb{R}
%EndExpansion
\right)  }^{2}/2  &  =\int\nolimits_{U}\widehat{\theta}_{2}\left(  t,x\right)
\partial_{t}\widehat{\theta}_{2}\left(  t,x\right)  dx\\
&  =\int\nolimits_{U}\alpha\left(  t,x\right)  \left(  1-w\left(  t,x\right)
\right)  \widehat{\theta}_{2}\left(  t,x\right)  dx-\int\nolimits_{U}%
\alpha\left(  t,x\right)  w\left(  t,x\right)  \widehat{\theta}_{2}^{2}\left(
t,x\right)  dx\\
&  \leq0
\end{align*}%
\begin{align*}
\partial_{t}\left\Vert \widehat{v}_{1}\left(  t,.\right)  \right\Vert
_{L^{2}\left(  U;%
%TCIMACRO{\U{211d} }%
%BeginExpansion
\mathbb{R}
%EndExpansion
\right)  }^{2}/2  &  =\int\nolimits_{U}\widehat{v}_{1}\left(  t,x\right)
\partial_{t}\widehat{v}_{1}\left(  t,x\right)  dx\\
&  =-\int\nolimits_{U}\beta\left(  t,x,-\widehat{v}_{1}\left(  t,x\right)
\right)  \Phi_{3}\left(  t,x,\widehat{\theta}_{1},-\widehat{v}_{1}\left(
t,x\right)  \right)  \widehat{v}_{1}\left(  t,x\right)  dx\\
&  \leq0
\end{align*}
and
\begin{align*}
d_{t}\left\Vert \widehat{v}_{2}\left(  t,.\right)  \right\Vert _{L^{2}\left(
U;%
%TCIMACRO{\U{211d} }%
%BeginExpansion
\mathbb{R}
%EndExpansion
\right)  }^{2}/2  &  =\int\nolimits_{U}\widehat{v}_{2}\left(  t,x\right)
\partial_{t}\widehat{v}_{2}\left(  t,x\right)  dx\\
&  \leq\int\nolimits_{U}\beta\left(  t,x,\widehat{v}_{2}\left(  t,x\right)
+v_{\max}\right)  \Phi_{3}\left(  t,x,\widehat{\theta}_{1},\widehat{v}%
_{2}\left(  t,x\right)  +v_{\max}\right)  \widehat{v}_{2}\left(  t,x\right)
dx\\
&  \leq0
\end{align*}
Using the above inequalities we get that functions $\widehat{\theta}%
_{i}\left(  t,.\right)  $ and $\widehat{v}_{i}\left(  t,.\right)  $
($i\in\left\{  1,2\right\}  $) are nonpositive on the sets $S_{i}^{\theta}$
and $S_{i}^{v}$ which are necessary empty. We conclude that $\forall t>0,$
$\left(  \widehat{\theta}\left(  t,.\right)  ,\widehat{v}\left(  t,.\right)
\right)  $ is valued in $\left[  0,1\right]  \times\left[  0,v_{\max}\right]
$.

Now let show the existence and the uniqueness of the solution of $\left(
\ref{ObserverPDE11}\right)  -\left(  \ref{ObserverPDE13}\right)  $. It
suffices to establish existence of a local solution to conclude the result.
Just as decomposition $\left(  \ref{DecomposedSolution}\right)  $ a solution
to $\left(  \ref{ObserverPDE11}\right)  -\left(  \ref{ObserverPDE12}\right)  $
can be written as the sum $%
%TCIMACRO{\dsum \nolimits_{n=0}^{\infty}}%
%BeginExpansion
{\displaystyle\sum\nolimits_{n=0}^{\infty}}
%EndExpansion
\widehat{\theta}_{n}\varphi_{n}$. Each $\widehat{\theta}_{n}$ satisifies the
realtion
\begin{align*}
\widehat{\theta}_{n}\left(  t\right)   &  =%
%TCIMACRO{\dsum \nolimits_{m=0}^{\infty}}%
%BeginExpansion
{\displaystyle\sum\nolimits_{m=0}^{\infty}}
%EndExpansion
\widehat{\theta}_{m}%
%TCIMACRO{\dint \nolimits_{U}}%
%BeginExpansion
{\displaystyle\int\nolimits_{U}}
%EndExpansion
\varphi_{m}\left(  x\right)  \varphi_{n}\left(  x\right)  dx\\
&  =\widehat{\theta}_{n}\left(  0\right)  +%
%TCIMACRO{\dint \nolimits_{0}^{t}}%
%BeginExpansion
{\displaystyle\int\nolimits_{0}^{t}}
%EndExpansion
\lambda_{n}\left(  s\right)  \widehat{\theta}_{n}\left(  s\right)  ds+%
%TCIMACRO{\dint \nolimits_{0}^{t}}%
%BeginExpansion
{\displaystyle\int\nolimits_{0}^{t}}
%EndExpansion%
%TCIMACRO{\dint \nolimits_{U}}%
%BeginExpansion
{\displaystyle\int\nolimits_{U}}
%EndExpansion
\varphi_{n}\left(  x\right)  \alpha\left(  s,x\right)  \left(  1-w\left(
s,x\right)  \widehat{\theta}\left(  s,x\right)  \right)  dxds\\
&  \ \ \ +%
%TCIMACRO{\dint \nolimits_{0}^{t}}%
%BeginExpansion
{\displaystyle\int\nolimits_{0}^{t}}
%EndExpansion%
%TCIMACRO{\dint \nolimits_{U}}%
%BeginExpansion
{\displaystyle\int\nolimits_{U}}
%EndExpansion
\varphi_{n}\left(  x\right)  K_{1}\left(  s,x\right)  \Phi_{1}\left(
s,x,\widehat{\theta}\left(  s,x\right)  ,\widehat{v}\left(  s,x\right)
\right)  dxds\\
&  \ \ \ +%
%TCIMACRO{\dint \nolimits_{0}^{t}}%
%BeginExpansion
{\displaystyle\int\nolimits_{0}^{t}}
%EndExpansion%
%TCIMACRO{\dint \nolimits_{U}}%
%BeginExpansion
{\displaystyle\int\nolimits_{U}}
%EndExpansion
\varphi_{n}\left(  x\right)  K_{2}\left(  s,x\right)  \Phi_{2}\left(
s,x,\widehat{\theta}\left(  s,x\right)  \right)  dxds
\end{align*}
Let $\ell^{2}\left(  \mathbb{R}\right)  $ denote the Hilbert space of
real-valued sequences $\left\{  s_{n}\right\}  _{n\in%
%TCIMACRO{\U{2115} }%
%BeginExpansion
\mathbb{N}
%EndExpansion
}$ such that $\left\langle s,s\right\rangle =\sum\nolimits_{n\in%
%TCIMACRO{\U{2115} }%
%BeginExpansion
\mathbb{N}
%EndExpansion
}s_{n}^{2}<\infty$ . We set
\[
S_{\theta}=\left\{  s\in\ell^{2}\left(  \mathbb{R}\right)  ;0\leq
\sum\nolimits_{n\in%
%TCIMACRO{\U{2115} }%
%BeginExpansion
\mathbb{N}
%EndExpansion
}s_{n}\leq1\right\}
\]
and
\[
S_{v}=\left\{  w\in L^{2}\left(  U;\mathbb{R}\right)  ;\forall x\in U,w\left(
x\right)  \in\left[  0,v_{\max}\right]  \right\}
\]
Let consider the linear operator $\mathcal{G}:D\left(  \mathcal{G}\right)
\rightarrow\ell^{2}\left(  \mathbb{R}\right)  \times L^{2}\left(
U;\mathbb{R}\right)  $ where $D\left(  \mathcal{G}\right)  \subseteq%
%TCIMACRO{\U{211d} }%
%BeginExpansion
\mathbb{R}
%EndExpansion
_{+}\times\ell^{2}\left(  \mathbb{R}\right)  \times L^{2}\left(
U;\mathbb{R}\right)  $ defined by
\[
\mathcal{G}_{n}^{1}\left(  t,y,z\right)  =\lambda_{n}\left(  t\right)  y_{n}-%
%TCIMACRO{\dint \nolimits_{U}}%
%BeginExpansion
{\displaystyle\int\nolimits_{U}}
%EndExpansion
\alpha\left(  t,x\right)  w\left(  t,x\right)  \varphi_{n}\left(  x\right)
\left\langle y,\varphi\left(  x\right)  \right\rangle dx-%
%TCIMACRO{\dint \nolimits_{U}}%
%BeginExpansion
{\displaystyle\int\nolimits_{U}}
%EndExpansion
K_{1}\left(  t,x\right)  \varphi_{n}\left(  x\right)  \delta\left(
\left\langle y,\varphi\left(  x\right)  \right\rangle \right)  \left\langle
y,\varphi\left(  x\right)  \right\rangle dx
\]%
\[
\mathcal{G}^{2}\left(  t,y,z\right)  =0
\]
$\mathcal{G}$ is the infinitesimal generator of the evolution system $E$
defined such as $\forall n\in%
%TCIMACRO{\U{2115} }%
%BeginExpansion
\mathbb{N}
%EndExpansion
$, $\forall t_{0}^{+}\leq s\leq t\leq t_{1}$,
\[
\left(  E^{1}\left(  s,t\right)  \left(  y,z\right)  \right)  _{n}=%
%TCIMACRO{\dint \nolimits_{U}}%
%BeginExpansion
{\displaystyle\int\nolimits_{U}}
%EndExpansion
\exp\left(  \int\nolimits_{s}^{t}\left(  \lambda_{n}\left(  \tau\right)
-\alpha\left(  \tau,x\right)  w\left(  \tau,x\right)  -K_{1}\left(
\tau,x\right)  \right)  d\tau\right)  \left\langle y,\varphi\left(  x\right)
\right\rangle \varphi\left(  x\right)  dx
\]%
\[
E^{2}\left(  s,t\right)  \left(  y,z\right)  =z
\]
Let also consider the operators $\mathcal{F}:D\left(  \mathcal{F}\right)
\cap\left(
%TCIMACRO{\U{211d} }%
%BeginExpansion
\mathbb{R}
%EndExpansion
_{+}\times S_{\theta}\times S_{v}\right)  \rightarrow\ell^{2}\left(
\mathbb{R}\right)  \times L^{2}\left(  U;\mathbb{R}\right)  $ where $D\left(
\mathcal{F}\right)  \subseteq%
%TCIMACRO{\U{211d} }%
%BeginExpansion
\mathbb{R}
%EndExpansion
_{+}\times\ell^{2}\left(  \mathbb{R}\right)  \times L^{2}\left(
U;\mathbb{R}\right)  $ defined by
\begin{align*}
\mathcal{F}_{n}^{1}\left(  t,y,z\right)   &  =%
%TCIMACRO{\dint \nolimits_{U}}%
%BeginExpansion
{\displaystyle\int\nolimits_{U}}
%EndExpansion
\varphi_{n}\left(  x\right)  \left(  \alpha\left(  t,x\right)  +K_{2}\left(
t,x\right)  \overline{\gamma}\left(  t,x,\theta\left(  t,x\right)  ,v\left(
t,x\right)  ,\rho\left(  t,x\right)  \right)  \left(  1-\rho\left(
t,x\right)  \right)  \right)  dx\\
&  \ \ \ +%
%TCIMACRO{\dint \nolimits_{U}}%
%BeginExpansion
{\displaystyle\int\nolimits_{U}}
%EndExpansion
K_{1}\left(  t,x\right)  \delta\left(  \left\langle y,\varphi\left(  x\right)
\right\rangle \right)  \frac{\left(  1+\varepsilon\right)  \left(  z-v\right)
\varphi_{n}\left(  x\right)  }{z}dx\\
&  \ \ \ +%
%TCIMACRO{\dint \nolimits_{U}}%
%BeginExpansion
{\displaystyle\int\nolimits_{U}}
%EndExpansion
K_{1}\left(  t,x\right)  \delta\left(  \left\langle y,\varphi\left(  x\right)
\right\rangle \right)  \frac{v\varphi_{n}\left(  x\right)  \left\langle
y,\varphi\left(  x\right)  \right\rangle }{z}dx\\
&  \ \ \ -%
%TCIMACRO{\dint \nolimits_{U}}%
%BeginExpansion
{\displaystyle\int\nolimits_{U}}
%EndExpansion
\varphi_{n}\left(  x\right)  K_{2}\left(  t,x\right)  \overline{\gamma}\left(
t,x,y,v\left(  t,x\right)  ,\rho\left(  t,x\right)  \right)  \left(
1-\rho\left(  t,x\right)  \right)  dx
\end{align*}%
\[
\mathcal{F}^{2}\left(  t,y,z\right)  =\beta\left(  t,x,\left\langle
y,\varphi\left(  x\right)  \right\rangle \right)  \Phi_{3}\left(
t,x,\left\langle y,\varphi\left(  x\right)  \right\rangle ,z\right)
\]
The function $\mathcal{F}$ is measurable with respect to the time $t$. We can
show as in the proof of proposition \ref{ExistenceUniciteOb1} that
$\mathcal{F}$ is Lipschitz continuous with respect to $\left(  y,z\right)  $
and there is a global and unique solution for the system $\left(
\ref{ObserverPDE11}\right)  -\left(  \ref{ObserverPDE13}\right)  $. Note that
$\forall\left(  t,x\right)  \in%
%TCIMACRO{\U{211d} }%
%BeginExpansion
\mathbb{R}
%EndExpansion
_{+}\times U$ the function $\beta\left(  t,x,.\right)  $ is Lipschitz
continuous and bounded on the set $\left[  0,1\right]  $ using the assumption
$\left(  \ref{Hypothesis_13}\right)  $\textbf{.}
\end{proof}

\begin{theorem}
\label{ObsLocPDE} Let consider the system $\left(  \ref{ObserverPDE11}\right)
-\left(  \ref{ObserverPDE13}\right)  $ and assume that $\widehat{v}\left(
0,.\right)  =v\left(  0,.\right)  $.

\begin{enumerate}
\item[$\left(  i\right)  $] If the functions $K_{1}$ and $K_{2}$ are
identically null almost everywhere on $U$ and
\begin{equation}
\inf\left\{  \alpha\left(  t,x\right)  ;t>0,x\in U\right\}  >0
\end{equation}
then $\widehat{\theta}$ is a globally exponentially asymptotically stable
observer for $\theta$.

\item[$\left(  ii\right)  $] If the function $K_{1}$ is identically null
almost everywhere on $U$ and $K_{2}$ is not identically null at least on a
subset of $U$ with positive measure then $\widehat{\theta}$ is a a locally
asymptotically stable observers for $\theta$. Moreover, if there is a positive
function $C_{\gamma}\in L_{loc}^{\infty}\left(
%TCIMACRO{\U{211d} }%
%BeginExpansion
\mathbb{R}
%EndExpansion
_{+}^{\ast}\times U;%
%TCIMACRO{\U{211d} }%
%BeginExpansion
\mathbb{R}
%EndExpansion
_{+}\right)  $ such that
\begin{equation}
\left\vert \overline{\gamma}\left(  t,x,y_{1},y_{2},y_{3}\right)
-\overline{\gamma}\left(  t,x,z_{1},z_{2},z_{3}\right)  \right\vert \geq
C_{\gamma}\left(  t,x\right)  \left\Vert y-z\right\Vert
\label{ConditionCoercive2}%
\end{equation}
then $\widehat{\theta}$ is a a globally exponentially asymptotically stable
observers for $\theta$.

\item[$\left(  iii\right)  $] If the function $K_{2}$ is identically null and
\begin{equation}
\inf\left\{  \frac{v+\left(  1+\varepsilon-\theta\right)  \partial_{\theta}%
v}{v}K_{1}\delta\left(  \widehat{\theta}_{1}\right)  +\alpha\left(
t,x\right)  w\left(  t,x\right)  ;\text{ }\left(  t,x\right)  \in%
%TCIMACRO{\U{211d} }%
%BeginExpansion
\mathbb{R}
%EndExpansion
_{+}\times U\right\}  >0\text{.} \label{ConditionStabilityPde1}%
\end{equation}
then $\widehat{\theta}$ is a a locally exponentially asymptotically stable
observers for $\theta$.

\item[$\left(  iv\right)  $] If the functions $K_{1}$ and $K_{2}$ are not
identically null, and
\begin{equation}
\inf\left\{  K_{1}\delta\left(  \widehat{\theta}\right)  \frac{v+\left(
1+\varepsilon-\theta\right)  \partial_{\theta}v}{v}+K_{2}\Phi_{2}\left(
t,x,\widehat{\theta}\right)  +\alpha\left(  t,x\right)  w\left(  t,x\right)
;\text{ }\left(  t,x\right)  \in%
%TCIMACRO{\U{211d} }%
%BeginExpansion
\mathbb{R}
%EndExpansion
_{+}\times U\right\}  >0\text{.} \label{ConditionStabilityPde2}%
\end{equation}
then $\widehat{\theta}$ is a a locally exponentially asymptotically stable
observers for $\theta$. Moreover, if the condition $\left(
\ref{ConditionCoercive2}\right)  $ is satisfied and
\begin{equation}
K_{2}\left(  t,x\right)  \left\vert \Phi_{2}\left(  t,\widehat{\theta}\left(
t,x\right)  \right)  \right\vert >K_{1}\left(  t,x\right)  \Phi_{1}\left(
t,\widehat{\theta}\left(  t,x\right)  ,\widehat{v}\left(  t,x\right)  \right)
\text{, }\forall\left(  t,x\right)  \in%
%TCIMACRO{\U{211d} }%
%BeginExpansion
\mathbb{R}
%EndExpansion
_{+}^{\ast}\times U \label{ConditionDominance2}%
\end{equation}
then $\widehat{\theta}$ is a a globally exponentially asymptotically stable
observers for $\theta$.
\end{enumerate}
\end{theorem}

\begin{proof}
Let $\widehat{e}=\theta-\widehat{\theta}$ be the error of estimation.

\begin{enumerate}
\item[$\left(  i\right)  $] The error $\widehat{e}$ satisfies
\begin{align}
d_{t}\left\Vert \widehat{e}\left(  t,.\right)  \right\Vert _{L^{2}\left(
U\right)  }^{2}/2  &  =-\int\nolimits_{U}\alpha\left(  t,x\right)  w\left(
t,x\right)  \widehat{e}^{2}\left(  t,x\right)  dx\\
&  \leq-\left\Vert \widehat{e}\left(  t,.\right)  \right\Vert _{L^{2}\left(
U\right)  }^{2}\underset{\left(  s,y\right)  \in%
%TCIMACRO{\U{211d} }%
%BeginExpansion
\mathbb{R}
%EndExpansion
_{+}\times U}{\inf}\left\{  \alpha\left(  s,y\right)  \right\} \nonumber
\end{align}
and
\begin{equation}
\left\Vert \widehat{e}\left(  t,.\right)  \right\Vert _{L^{2}\left(  U\right)
}^{2}\leq\exp\left(  -2t\underset{\left(  s,y\right)  \in%
%TCIMACRO{\U{211d} }%
%BeginExpansion
\mathbb{R}
%EndExpansion
_{+}\times U}{\inf}\left\{  \alpha\left(  s,y\right)  \right\}  \right)
\left\Vert \widehat{e}\left(  0,.\right)  \right\Vert _{L^{2}\left(  U\right)
}^{2}\text{.}%
\end{equation}
We deduce the results.

\item[$\left(  ii\right)  $] The error $\widehat{e}$ satisfies
\begin{align}
d_{t}\left\Vert \widehat{e}\left(  t,.\right)  \right\Vert _{L^{2}\left(
U\right)  }^{2}/2  &  =-\int\nolimits_{U}\alpha\left(  t,x\right)  w\left(
t,x\right)  \widehat{e}^{2}dx-\int\nolimits_{U}\left\langle A\left(
t,x\right)  \nabla\widehat{e},\nabla\widehat{e}\right\rangle dx\\
&  \text{ \ \ \ }-\int\nolimits_{U}K_{2}\left(  t,x\right)  \Phi_{2}\left(
t,x,\widehat{\theta}\right)  \widehat{e}^{2}\left(  t,x\right)  dx\\
&  \leq-\int\nolimits_{U}\frac{v+\left(  1+\varepsilon-\theta\right)
\partial_{\theta}v}{v}K_{1}\left(  t,x\right)  \delta\left(  \widehat{\theta
}\right)  \widehat{e}^{2}dx-\int\nolimits_{U}\alpha\left(  t,x\right)
w\left(  t,x\right)  \widehat{e}^{2}dx\nonumber\\
&  \text{ \ \ \ }-\int\nolimits_{U}K_{2}\left(  t,x\right)  \left(
1-\rho\right)  \left(  \overline{\gamma}\left(  t,x,\theta,v,\rho\right)
-\overline{\gamma}\left(  t,x,\widehat{\theta},v,\rho\right)  \right)
\widehat{e}\left(  t,x\right)  dx\nonumber
\end{align}
and
\[
\left\Vert \widehat{e}\left(  t,.\right)  \right\Vert _{L^{2}\left(  U\right)
}^{2}\leq\exp\left(  -2t\underset{\left(  s,y\right)  \in%
%TCIMACRO{\U{211d} }%
%BeginExpansion
\mathbb{R}
%EndExpansion
_{+}\times U}{\inf}\left\{  \alpha\left(  s,y\right)  \right\}  \right)
\left\Vert \widehat{e}\left(  0,.\right)  \right\Vert _{L^{2}\left(  U\right)
}^{2}\text{.}%
\]
Moreover, if the condition $\left(  \ref{ConditionCoercive2}\right)  $ is
satisfied then
\begin{align*}
d_{t}\left\Vert \widehat{e}\left(  t,.\right)  \right\Vert _{L^{2}\left(
U\right)  }^{2}  &  \leq-\int\nolimits_{U}2\alpha\left(  t,x\right)  w\left(
t,x\right)  \widehat{e}^{2}\left(  t,x\right)  dx-\int\nolimits_{U}%
2K_{2}\left(  t,x\right)  \left(  1-\rho\left(  t,x\right)  \right)
C_{\gamma}\left(  t\right)  \widehat{e}^{2}\left(  t,x\right)  dx\\
&  \leq-2\left\Vert \widehat{e}\left(  t,.\right)  \right\Vert _{L^{2}\left(
U\right)  }^{2}\underset{\left(  s,y\right)  \in%
%TCIMACRO{\U{211d} }%
%BeginExpansion
\mathbb{R}
%EndExpansion
_{+}\times U}{\inf}\left\{  \alpha\left(  s,y\right)  w\left(  s,y\right)
+K_{2}\left(  s,y\right)  \left(  1-\rho\left(  s,y\right)  \right)
C_{\gamma}\left(  s,y\right)  \right\}
\end{align*}
and
\[
\left\Vert \widehat{e}\left(  t,.\right)  \right\Vert _{L^{2}\left(  U\right)
}^{2}\leq\exp\left(  -2t\underset{\left(  s,y\right)  \in%
%TCIMACRO{\U{211d} }%
%BeginExpansion
\mathbb{R}
%EndExpansion
_{+}\times U}{\inf}\left\{  \alpha\left(  s,y\right)  w\left(  s,y\right)
+K_{2}\left(  s,y\right)  \left(  1-\rho\left(  s,y\right)  \right)
C_{\gamma}\left(  s,y\right)  \right\}  \right)  \left\Vert \widehat{e}\left(
0,.\right)  \right\Vert _{L^{2}\left(  U\right)  }^{2}\text{.}%
\]
Since the function $C_{\gamma}$ is positive the results holds.

\item[$\left(  iii\right)  $] The error $\widehat{e}$ satisfies%
\begin{align*}
d_{t}\left\Vert \widehat{e}\left(  t,.\right)  \right\Vert _{L^{2}\left(
U\right)  }^{2}/2  &  =-\int\nolimits_{U}\alpha\left(  t,x\right)  w\left(
t,x\right)  \widehat{e}^{2}\left(  t,x\right)  dx-\int\nolimits_{U}%
\left\langle A\left(  t,x\right)  \nabla\widehat{e}^{2}\left(  t,x\right)
,\nabla\widehat{e}^{2}\left(  t,x\right)  \right\rangle dx\\
&  \,\,\,\,\,\,\,-\int\nolimits_{U}K_{1}\left(  t,x\right)  \Phi_{1}\left(
t,x,\widehat{\theta}_{1}\left(  t,x\right)  ,\widehat{v}_{1}\left(
t,x\right)  \right)  \widehat{e}^{2}\left(  t,x\right)  dx\\
&  \leq-\int\nolimits_{U}\frac{v+\left(  1+\varepsilon-\theta\right)
\partial_{\theta}v}{v}K_{1}\left(  t,x\right)  \delta\left(  \widehat{\theta
}_{1}\right)  \widehat{e}^{2}\left(  t,x\right)  dx\\
&  \,\,\,\,\,\,\,-\int\nolimits_{U}\alpha\left(  t,x\right)  w\left(
t,x\right)  \widehat{e}^{2}\left(  t,x\right)  dx\\
&  \approx-\int\nolimits_{U}\frac{K_{1}\left(  t,x\right)  \delta\left(
\widehat{\theta}\right)  \beta\left(  t,\theta\right)  \left(  \eta v_{\max
}\left(  1+\varepsilon-\theta\right)  -v\right)  }{\alpha\eta vv_{\max}\left(
1-\theta w\right)  }\widehat{e}^{2}\left(  t,x\right)  dx\\
&  \,\,\,\,\,\,\,-\int\nolimits_{U}\left(  \alpha\left(  t\right)  w\left(
t\right)  +K_{1}\left(  t\right)  \delta\left(  \widehat{\theta}\right)
\right)  \widehat{e}^{2}\left(  t,x\right)  dx\text{.}%
\end{align*}
Using the condition $\left(  \ref{ConditionStabilityPde1}\right)  $ we deduce
that $\widehat{\theta}$ is a locally exponentially asymptotically stable
observers for $\theta$.

\item[$\left(  iv\right)  $] Using the same arguments developped in $\left(
ii\right)  $ and $\left(  iii\right)  $ we get that $\widehat{\theta}$ is a
locally exponentially asymptotically stable observers for $\theta$. With the
conditions $\left(  \ref{ConditionCoercive2}\right)  $ and $\left(
\ref{ConditionDominance2}\right)  $ we are sure that even if the term with
$\Phi_{1}$ introduces an unstability, the stability is maintained by the term
with $\Phi_{2}$.
\end{enumerate}
\end{proof}

\subsection{Numerical evaluation of the spatial model's
observers\label{ObserversEvalutionNum2}}

In this section we present numerical simulations in order to check the
effectiveness of the observers we studied theoretically in the subsection
\ref{SubsectionDesignTheoric2}. The duration of the simulations is the year.
Parameters are taken following \cite{fotsa} and generalize those in the
subsection \ref{ObserversEvalutionNum1}. We assume a anisotropic radial
control strategy $u$ given for every $\left(  t,x\right)  \in%
%TCIMACRO{\U{211d} }%
%BeginExpansion
\mathbb{R}
%EndExpansion
_{+}\times U$,
\begin{equation}
u(t,x)=\sin^{2}\left(  \left\Vert M\left(  x-x_{0}\right)  \right\Vert
^{2}\right)  \sin^{2}\left(  \omega_{1}\left(  t-\varphi_{1}\right)
^{2}\right)  \exp\left(  -\omega_{2}\left(  t-\varphi_{2}\right)  ^{2}\right)
\text{.}%
\end{equation}
The matrix $M$ introduces a geometrical anisotropy in the distribution of $u$.
The functions $\alpha,\beta$ and $\gamma$ are taken such as
\begin{equation}
\alpha\left(  t,x\right)  =p_{1}\left(  t,x\right)  +b_{1}q_{1}\left(
x\right)  \left(  1-\cos\left(  c_{1}t\right)  \right)  \left(  t-d_{1}%
\right)  ^{2}\text{, }\left(  t,x\right)  \in%
%TCIMACRO{\U{211d} }%
%BeginExpansion
\mathbb{R}
%EndExpansion
_{+}\times U
\end{equation}%
\begin{equation}
\beta\left(  t,x,y\right)  =b_{2}q_{2}\left(  x\right)  \left(  1-\cos\left(
c_{2}t\right)  \right)  \left(  t-d_{2}\right)  ^{2}p_{2}\left(  y\right)
\text{, }\left(  t,x,y\right)  \in%
%TCIMACRO{\U{211d} }%
%BeginExpansion
\mathbb{R}
%EndExpansion
_{+}\times U\times%
%TCIMACRO{\U{211d}}%
%BeginExpansion
\mathbb{R}%
%EndExpansion
\end{equation}
and
\begin{equation}
\gamma\left(  t,x,y_{1},y_{2},y_{3}\right)  =b_{3}q_{3}\left(  x\right)
\left(  1-\cos\left(  c_{3}t\right)  \right)  \left(  t-d_{3}\right)
^{2}\left(  y_{1}-\kappa y_{3}\right)  y_{2}\text{, }\left(  t,x,y\right)  \in%
%TCIMACRO{\U{211d} }%
%BeginExpansion
\mathbb{R}
%EndExpansion
_{+}\times U\times%
%TCIMACRO{\U{211d} }%
%BeginExpansion
\mathbb{R}
%EndExpansion
^{3}\text{.}%
\end{equation}
$p_{1}$ is a nonnegative function of the time, $p_{2}$ is a positive function
and for every $i$ in $\left\{  1,2,3\right\}  $, $q_{i}$ is a real anisotropic
radial nonnegative function of the space variable as the control $u$. For
every $i$ in $\left\{  1,2,3\right\}  $, $b_{i}$, $c_{i}$, and $d_{i}$ are
positive coefficients corresponding respectively to the maximal amplitude, the
pulsation and the global maximun of $\alpha,\beta$ and $\gamma$. For every $i$
in $\left\{  1,2,3\right\}  $, $q_{i}$ is a $\left[  0,1\right]  $-valued
function of the space variable. The matrix $A$ in equation $\left(
\ref{ModelSpace1}\right)  $ is assumed to be equal to $10^{-2}I$ because the
diffusion is relatively slow. The functions $K_{1}$ and $K_{2}$ are taken
constant and are similar with $k_{1}$ and $k_{2}$ defined in subsection
\ref{ObserversEvalutionNum1}. The rest of the parameters are the same
described in the table \ref{TableParam1}. The initial conditions are also
taken following the subsection \ref{ObserversEvalutionNum1} and are constant
with respect to the space variable in order to easily fulfill the boundary
conditions $\left(  \ref{ModelSpace2}\right)  $. The following table specifies
some assumed values of additional parameters: For every $i\in\left\{
1;2;3\right\}  $,

\begin{table}[ptbh]
\centering%
\begin{tabular}
[c]{|c|c|c|c|c|c|}\hline
\textbf{Parameters} & \textbf{Values} & \textbf{Source} & \textbf{Parameters}
& \textbf{Values} & \textbf{Source}\\\hline
$q_{i}\left(  x\right)  $ & $\frac{\sin^{2}\left(  \left\Vert M_{i}\left(
x-x_{i}\right)  \right\Vert ^{2}\right)  +1}{2}$ & Assumed & $M,M_{i}$ &
randomly generated & \\
$x_{i}$ & $0$ & Assumed &  & $5\ast\operatorname{rand}\left(  3,3\right)  $ &
Assumed\\\hline
\end{tabular}
%title of Table
\caption{Simulation parameters for the spatial model}%
\label{TableParam2}%
\end{table}

In the remainder of the subsection we present simulations of the observers and
their estimation absolute relative error dynamics. To keep the representation
simple we just plot the spatial minimum values, the spatial average values and
the spatial maximum values. Instead of the estimation relative error we
consider the estimation absolute relative error to avoid compensations when
averaging. The initial condition $\rho\left(  0,.\right)  $ is always taken
equal to $\theta\left(  0,.\right)  $. Indeed, following simulations in
subsection \ref{ObserversEvalutionNum1} the performances of the observers
decrease with respect to $\rho\left(  0,.\right)  $ and on the other hand, it
seems realistic to take $\rho\left(  0,.\right)  $ less than $\theta\left(
0,.\right)  $.

\subsubsection{Simulations with $\theta\left(  0,.\right)  =v\left(
0,.\right)  =0.05$}

The figures \ref{figInhPDE11} and \ref{figErrPDE11} show the behaviour of the
observers when the observation is started early. As we can see the global
performance of each observer is relatively good but the case $k_{1}=0$ and
$k_{2}=10^{3}$ seems to be the better one.

\begin{figure}[ptbh]
\centering
\begin{tabular}
[c]{cc}%
\raisebox{-0cm}{\includegraphics[
			natheight=12.171500cm,
			natwidth=16.138599cm,
			height=6.1549cm,
			width=8.1385cm
			]{EstimInhRatPDE0011.jpg}} & \raisebox{-0cm}{\includegraphics[
			natheight=12.171500cm,
			natwidth=16.138599cm,
			height=6.1549cm,
			width=8.1385cm
			]{EstimInhRatPDE0111.jpg}}\\
$k_{1}=k_{2}=0$ & $k_{1}=0$ and $k_{2}=10^{3}$\\
\raisebox{-0cm}{\includegraphics[
			natheight=12.171500cm,
			natwidth=16.138599cm,
			height=6.1549cm,
			width=8.1385cm
			]{EstimInhRatPDE1011.jpg}} & \raisebox{-0cm}{\includegraphics[
			natheight=12.171500cm,
			natwidth=16.138599cm,
			height=6.1549cm,
			width=8.1385cm
			]{EstimInhRatPDE1111.jpg}}\\
$k_{1}=10^{3}$ and $k_{2}=0$ & $k_{1}=10^{3}$ and $k_{2}=10^{3}$%
\end{tabular}
%title of Figure
\caption{Inhibition rate and observers for $\theta\left(  0,.\right)
=v\left(  0,.\right)  =0.05$.}%
\label{figInhPDE11}%
\end{figure}

\begin{figure}[ptbh]
\centering
\begin{tabular}
[c]{cc}%
\raisebox{-0cm}{\includegraphics[
			natheight=12.171500cm,
			natwidth=16.138599cm,
			height=6.1527cm,
			width=8.1385cm
			]{ReAbsErrPDE0011.jpg}} & \raisebox{-0cm}{\includegraphics[
			natheight=12.171500cm,
			natwidth=16.138599cm,
			height=6.1527cm,
			width=8.1385cm
			]{ReAbsErrPDE0111.jpg}}\\
$k_{1}=k_{2}=0$ & $k_{1}=0$ and $k_{2}=10^{3}$\\
\raisebox{-0cm}{\includegraphics[
			natheight=12.171500cm,
			natwidth=16.138599cm,
			height=6.1527cm,
			width=8.1385cm
			]{ReAbsErrPDE1011.jpg}} & \raisebox{-0cm}{\includegraphics[
			natheight=12.171500cm,
			natwidth=16.138599cm,
			height=6.1527cm,
			width=8.1385cm
			]{ReAbsErrPDE1111.jpg}}\\
$k_{1}=10^{3}$ and $k_{2}=0$ & $k_{1}=10^{3}$ and $k_{2}=10^{3}$%
\end{tabular}
%title of Figure
\caption{Relative absolute estimation error for $\theta\left(  0,.\right)
=v\left(  0,.\right)  =0.05$.}%
\label{figErrPDE11}%
\end{figure}\newpage

\subsubsection{Simulations with $\theta\left(  0,.\right)  =0.05$ and
$v\left(  0,.\right)  =0.5$}

The figures \ref{figInhPDE13} and \ref{figErrPDE13} correspond to an
observation starting at the beginning of the disease but with a well
developped berry. We have better performances when $k_{2}=10^{3}$. With
$k_{1}=10^{3}$ and $k_{2}=0$ we notice an unstability behaviour in figure
\ref{figErrPDE13} although the error takes the null value. That may be due to
the constant sign of the function $\phi_{1}$.

\begin{figure}[ptbh]
\centering
\begin{tabular}
[c]{cc}%
\raisebox{-0cm}{\includegraphics[
			natheight=12.171500cm,
			natwidth=16.138599cm,
			height=6.1549cm,
			width=8.1385cm
			]{EstimInhRatPDE0012.jpg}} & \raisebox{-0cm}{\includegraphics[
			natheight=12.171500cm,
			natwidth=16.138599cm,
			height=6.1549cm,
			width=8.1385cm
			]{EstimInhRatPDE0112.jpg}}\\
$k_{1}=k_{2}=0$ & $k_{1}=0$ and $k_{2}=10^{3}$\\
\raisebox{-0cm}{\includegraphics[
			natheight=12.171500cm,
			natwidth=16.138599cm,
			height=6.1549cm,
			width=8.1385cm
			]{EstimInhRatPDE1012.jpg}} & \raisebox{-0cm}{\includegraphics[
			natheight=12.171500cm,
			natwidth=16.138599cm,
			height=6.1549cm,
			width=8.1385cm
			]{EstimInhRatPDE1112.jpg}}\\
$k_{1}=10^{3}$ and $k_{2}=0$ & $k_{1}=10^{3}$ and $k_{2}=10^{3}$%
\end{tabular}
%title of Figure
\caption{Inhibition rate and observers for $\theta\left(  0,.\right)  =0.05$
and $v\left(  0,.\right)  =0.5$.}%
\label{figInhPDE13}%
\end{figure}

\begin{figure}[ptbh]
\centering
\begin{tabular}
[c]{cc}%
\raisebox{-0cm}{\includegraphics[
			natheight=12.171500cm,
			natwidth=16.138599cm,
			height=6.1549cm,
			width=8.1385cm
			]{ReAbsErrPDE0012.jpg}} & \raisebox{-0cm}{\includegraphics[
			natheight=12.171500cm,
			natwidth=16.138599cm,
			height=6.1549cm,
			width=8.1385cm
			]{ReAbsErrPDE0112.jpg}}\\
$k_{1}=k_{2}=0$ & $k_{1}=0$ and $k_{2}=10^{3}$\\
\raisebox{-0cm}{\includegraphics[
			natheight=12.171500cm,
			natwidth=16.138599cm,
			height=6.1549cm,
			width=8.1385cm
			]{ReAbsErrPDE1012.jpg}} & \raisebox{-0cm}{\includegraphics[
			natheight=12.171500cm,
			natwidth=16.138599cm,
			height=6.1549cm,
			width=8.1385cm
			]{ReAbsErrPDE1112.jpg}}\\
$k_{1}=10^{3}$ and $k_{2}=0$ & $k_{1}=10^{3}$ and $k_{2}=10^{3}$%
\end{tabular}
%title of Figure
\caption{Relative absolute estimation error for $\theta\left(  0,.\right)
=0.05$ and $v\left(  0,.\right)  =0.5$.}%
\label{figErrPDE13}%
\end{figure}\newpage

\subsubsection{Simulations with $\theta\left(  0,.\right)  =0.75$ and
$v\left(  0,.\right)  =0.05$}

In the figures \ref{figInhPDE41} and \ref{figErrPDE41} the inhibition rate is
already high while the berry is little. As expected the worst estimation
corresponds to $k_{1}=k_{2}=0$. The best estimations are still made when
$k_{2}=10^{3}$. However, when $k_{1}=10^{3}$, the same unstability behaviour
in figure \ref{figErrPDE13} appears.

\begin{figure}[ptbh]
\centering
\begin{tabular}
[c]{cc}%
\raisebox{-0cm}{\includegraphics[
			natheight=12.171500cm,
			natwidth=16.138599cm,
			height=6.1549cm,
			width=8.1385cm
			]{EstimInhRatPDE0021.jpg}} & \raisebox{-0cm}{\includegraphics[
			natheight=12.171500cm,
			natwidth=16.138599cm,
			height=6.1549cm,
			width=8.1385cm
			]{EstimInhRatPDE0121.jpg}}\\
$k_{1}=k_{2}=0$ & $k_{1}=0$ and $k_{2}=10^{3}$\\
\raisebox{-0cm}{\includegraphics[
			natheight=12.171500cm,
			natwidth=16.138599cm,
			height=6.1549cm,
			width=8.1385cm
			]{EstimInhRatPDE1021.jpg}} & \raisebox{-0cm}{\includegraphics[
			natheight=12.171500cm,
			natwidth=16.138599cm,
			height=6.1549cm,
			width=8.1385cm
			]{EstimInhRatPDE1121.jpg}}\\
$k_{1}=10^{3}$ and $k_{2}=0$ & $k_{1}=10^{3}$ and $k_{2}=10^{3}$%
\end{tabular}
%title of figure
\caption{Inhibition rate and observers for $\theta\left(  0,.\right)  =0.75$
and $v\left(  0,.\right)  =0.05$.}%
\label{figInhPDE41}%
\end{figure}

\begin{figure}[ptbh]
\centering
\begin{tabular}
[c]{cc}%
\raisebox{-0cm}{\includegraphics[
			natheight=12.171500cm,
			natwidth=16.138599cm,
			height=6.1527cm,
			width=8.1385cm
			]{ReAbsErrPDE0021.jpg}} & \raisebox{-0cm}{\includegraphics[
			natheight=12.171500cm,
			natwidth=16.138599cm,
			height=6.1527cm,
			width=8.1385cm
			]{ReAbsErrPDE0121.jpg}}\\
$k_{1}=k_{2}=0$ & $k_{1}=0$ and $k_{2}=10^{3}$\\
\raisebox{-0cm}{\includegraphics[
			natheight=12.171500cm,
			natwidth=16.138599cm,
			height=6.1527cm,
			width=8.1385cm
			]{ReAbsErrPDE1021.jpg}} & \raisebox{-0cm}{\includegraphics[
			natheight=12.171500cm,
			natwidth=16.138599cm,
			height=6.1527cm,
			width=8.1385cm
			]{ReAbsErrPDE1121.jpg}}\\
$k_{1}=10^{3}$ and $k_{2}=0$ & $k_{1}=10^{3}$ and $k_{2}=10^{3}$%
\end{tabular}
%title of Table
\caption{Relative absolute estimation error for $\theta\left(  0,.\right)
=0.75$ and $v\left(  0,.\right)  =0.05$. }%
\label{figErrPDE41}%
\end{figure}\newpage

\subsubsection{Simulations with $\theta\left(  0,.\right)  =0.75$ and
$v\left(  0,.\right)  =0.5$}

For the figures \ref{figInhPDE43} and \ref{figErrPDE43} the inhibition rate is
already high and the fruit is relatively mature when the observation starts.
The worst case are still when $k_{2}=0$ and the best are with $k_{2}=10^{3}$.

\begin{figure}[ptbh]
\centering
\begin{tabular}
[c]{cc}%
\raisebox{-0cm}{\includegraphics[
			natheight=12.171500cm,
			natwidth=16.138599cm,
			height=6.1549cm,
			width=8.1385cm
			]{EstimInhRatPDE0022.jpg}} & \raisebox{-0cm}{\includegraphics[
			natheight=12.171500cm,
			natwidth=16.138599cm,
			height=6.1549cm,
			width=8.1385cm
			]{EstimInhRatPDE0122.jpg}}\\
$k_{1}=k_{2}=0$ & $k_{1}=0$ and $k_{2}=10^{3}$\\
\raisebox{-0cm}{\includegraphics[
			natheight=12.171500cm,
			natwidth=16.138599cm,
			height=6.1549cm,
			width=8.1385cm
			]{EstimInhRatPDE1022.jpg}} & \raisebox{-0cm}{\includegraphics[
			natheight=12.171500cm,
			natwidth=16.138599cm,
			height=6.1549cm,
			width=8.1385cm
			]{EstimInhRatPDE1122.jpg}}\\
$k_{1}=10^{3}$ and $k_{2}=0$ & $k_{1}=10^{3}$ and $k_{2}=10^{3}$%
\end{tabular}
%title of figure
\caption{Inhibition rate and observers for $\theta\left(  0,.\right)  =0.75$
and $v\left(  0,.\right)  =0.5$.}%
\label{figInhPDE43}%
\end{figure}

\begin{figure}[ptbh]
\centering
\begin{tabular}
[c]{cc}%
\raisebox{-0cm}{\includegraphics[
			natheight=12.171500cm,
			natwidth=16.138599cm,
			height=6.1527cm,
			width=8.1385cm
			]{ReAbsErrPDE0022.jpg}} & \raisebox{-0cm}{\includegraphics[
			natheight=12.171500cm,
			natwidth=16.138599cm,
			height=6.1527cm,
			width=8.1385cm
			]{ReAbsErrPDE0122.jpg}}\\
$k_{1}=k_{2}=0$ & $k_{1}=0$ and $k_{2}=10^{3}$\\
\raisebox{-0cm}{\includegraphics[
			natheight=12.171500cm,
			natwidth=16.138599cm,
			height=6.1527cm,
			width=8.1385cm
			]{ReAbsErrPDE1022.jpg}} & \raisebox{-0cm}{\includegraphics[
			natheight=12.171500cm,
			natwidth=16.138599cm,
			height=6.1527cm,
			width=8.1385cm
			]{ReAbsErrPDE1122.jpg}}\\
$k_{1}=10^{3}$ and $k_{2}=0$ & $k_{1}=10^{3}$ and $k_{2}=10^{3}$%
\end{tabular}
%title of figure
\caption{Relative absolute estimation error for $\theta\left(  0,.\right)
=0.75$ and $v\left(  0,.\right)  =0.5$.}%
\label{figErrPDE43}%
\end{figure}\newpage

\section{Conclusion and prospects\label{Conclusion}}

We have studied four observers for the dynamics of anthracnose disease. Our
constructions were based on the models given in \cite{fotsa}. We consider two
models: a spatial version and a non spatial one stated respectively by the
equations $\left(  \ref{ModelSpace1}\right)  -\left(  \ref{ModelSpace6}%
\right)  $ and $\left(  \ref{ModelIntraHote1}\right)  -\left(
\ref{ModelIntraHote4}\right)  $. Some changes have been made in the original
form of those models in order to be more realistic. Precisely, we have changed
the equations modelling the dynamics of the berry volume $\left(  v\right)  $
and the rot volume $\left(  v_{r}\right)  $. Indeed, before the launch of the
disease, the berry has to reach a minimal volume. On the other hand the rot
volume shall never be greater than total volume. We found that the system
describing the disease dynamics can be viewed as the natural observer of the
system even the performances are poor. We have studied theoretically two
observers which gives fairly good results looking at the relative error.
Several simulations have been carried out to assess the effectiveness of our
proposed observers. That work has been done for the spatial and the non
spatial models.

Although proposed observers seem to display fairly good results, we think that
a stochastic approach could be introduced in the estimation process. Indeed,
the parameters of the models are submitted to several stochastic variations.
Moreover, the measures of $v$ and $v_{r}$ also contains errors. A mean to take
into account those errors could consist in adding noises in our models and
proceed to the stochastic filtering by different existing tools. That aspect
is studied in on going works of some of the authors of the current paper.

\end{document}